\RequirePackage{fix-cm}
\RequirePackage{graphics}
\RequirePackage{color}
\RequirePackage{mathrsfs}
\RequirePackage{amssymb}
\RequirePackage{amsmath}
\RequirePackage{boxedminipage}
\RequirePackage{stmaryrd}
\RequirePackage{multirow}
\RequirePackage{booktabs}
\RequirePackage{accents}
\RequirePackage{bm}
\RequirePackage{tikz}
\usetikzlibrary{arrows,shapes,chains}
\documentclass[smallextended]{svjour3}       % onecolumn (second format)
\smartqed  % flush right qed marks, e.g. at end of proof
\usepackage{graphicx}
\newcommand{\RR}{\mathbf R}
\newtheorem{assumption}{Assumption}
\begin{document}
\title{First-Order Primal-Dual Method for Nonlinear Convex Cone Programming
\thanks{Acknowledgments: this research was supported by NSFC: 71471112, 71871140.}
}

\titlerunning{First-Order Primal-Dual Method for NCCP}        % if too long for running head

\author{Lei Zhao         \and
        Daoli Zhu %etc.
}


\institute{Lei Zhao \at
           Antai College of Economics and Management and Sino-US Global Logistics Institute, Shanghai Jiao Tong University, 200030 Shanghai, China \\
           %Tel.: +086-21-62932218\\
           \email{l.zhao@sjtu.edu.cn}           %  \\
           \and
           Daoli Zhu \at
           Antai College of Economics and Management and Sino-US Global Logistics Institute, Shanghai Jiao Tong University, 200030 Shanghai, China \\
           Tel.: +086-21-62932218\\
           \email{dlzhu@sjtu.edu.cn}
}

\date{Received: date/ Accepted: date}
%The correct dates will be entered by the editor

\maketitle

\begin{abstract}
Nonlinear Convex Cone Programming (NCCP) problems are important and have many practical applications. In this paper, we introduces a flexible first-order primal-dual algorithm called the Variant Auxiliary Problem Principle (VAPP) for solving NCCP problems when the objective function and constraints are smooth  and may be nonsmooth. Each iteration of VAPP generates a nonlinear approximation to the primal problem of an augmented Lagrangian method. The approximation incorporates both linearization and a variable distance-like function, and then the iterations of VAPP provide one decomposition property for NCCP. Motivated by recent applications in big data analysis, there has been an explosive growth in interest in the convergence rate analysis of parallel computing algorithms for large scale optimization problem. This paper proposes an iteration-based error bound and linear convergence of VAPP. Some verifiable sufficient conditions of this error bound are also discussed. For the general convex case (without error bound), we establish $O(1/t)$ convergence rate for primal suboptimality, feasibility and dual suboptimality. By adaptively setting in parameters at different iterations, we show an $O(1/t^2)$ rate for the strongly convex case. We further present Forward-Backward Splitting (FBS) formulation of VAPP method and establish the connection between VAPP and other primal-dual splitting methods. Finally, we discuss some issues in the implementation of VAPP.
\end{abstract}
\section{Introduction}\label{intro}
\indent In this paper, we consider Nonlinear Convex Cone Programming (NCCP):
\begin{equation}\label{Prob:general-function}
\begin{array}{lll}
\mbox{(P):}  &\min       & G(u)+J(u)      \\
             &\rm {s.t}  & \Theta(u)=\Omega(u)+\Phi(u)\in -\mathbf{C} \\
             &           & u\in \mathbf{U}
\end{array}
\end{equation}
where $G$ is a convex smooth function on the closed convex set $\mathbf{U}\subset \RR^{n}$, and $J$ is a convex, possibly nonsmooth function on $\mathbf{U}\subset \RR^{n}$. $\Omega$ is a smooth and $\Phi$ is a possibly nonsmooth mapping from $\RR^{n}$ to $\RR^{m}$. $\Omega(u)$ and $\Phi(u)$ are $\mathbf{C}$-convex and $\mathbf{C}$ is a nonempty closed convex cone in $\RR^{m}$ with vertex at the origin, that is, $\alpha\mathbf{C}+\beta\mathbf{C}\subset \mathbf{C}$, for $\alpha,\beta\geq 0$. It is obvious that when $\mathring{\mathbf{C}}$ (the interior of $\mathbf{C}$) is nonempty, the constraint $\Theta(u)\in -\mathbf{C}$ corresponds to an inequality constraint. The case $\mathbf{C}=\{0\}$ corresponds to an equality constraint. $\mathbf{C}^{*}$ denotes the conjugate cone i.e. $\mathbf{C}^*=\{y|\langle y,x\rangle\geq 0, \forall x\in\mathbf{C}\}$.\\
\indent NCCP is an important and challenging problem class from the viewpoint of optimization theory. Nonlinear programming, nonlinear semi-infinite programming (Goberna and L\'opez~\cite{SIP1998}, L\'opez and Still~\cite{SIP2007}, Shapiro~\cite{SIP2009}), and nonlinear second-order cone programming (Alizadeh and Goldfarb~\cite{SOCP03}, Fukushima et al.~\cite{Fukushima2012,Fukushima2009,Fukushima2007}, Yamashita and Yabe~\cite{Yamashita2009}) are special classes of NCCP.\\
\indent Furthermore, NCCP has numerous applications such as robust optimization (Ben-Tal and Nemirovski~\cite{Bental1998}, Ben-Tal et al.~\cite{Bental09}), finite impulse-response filter design (Lobo et al.~\cite{FIR1998}, Wu et al.~\cite{FIR1996}), total variation denoising and compressed sensing (Cand\`es et al.~\cite{CandesRombergTao06} and Donoho~\cite{Donoho-2006}), resource allocation (Patriksson~\cite{RA2008}, Patriksson and Str\"omberg~\cite{RA2015}), and so on.\\
\indent For general convex programming, the augmented Lagrangian method can overcome the instability and nondifferentiability of the Lagrangian dual function. Furthermore, the augmented Lagrangian of a constrained convex program has the same solution set as the original constrained convex program. The augmented Lagrangian approach for equality-constrained optimization problems was introduced in Hestenes~\cite{Hestenes1969} and Powell~\cite{Powell1969}, and then extended to inequality-constrained problems by Buys~\cite{Buys1972}. Theoretical properties of the augmented Lagrangian duality method on a finite-dimensional space were investigated by Rockafellar~\cite{Rock76}. Some properties of the augmented Lagrangian in finite-dimensional cone-constrained optimization are provided by Shapiro and Sun~\cite{ShapiroSun2004}.\\
\indent Although the augmented Lagrangian approach has several advantages, it does not preserve separability, even when the initial problem is separable. One way to decompose the augmented Lagrangian is Alternating Direction Method of Multipliers (ADMM) (Fortin and Glowinski~\cite{ADMM1983}). ADMM applies a well-known Gauss-Seidel-like minimization strategy. Because of the excellent numerical performance, some algorithmic tools are developed based on ADMM. (e.g.~\cite{SolverADM}) Another way to overcome this difficulty is the Auxiliary Problem Principle of Augmented Lagrangian methods (APP-AL) (Cohen and Zhu~\cite{CohenZ}), which is a fairly general first-order primal-dual parallel decomposition method based on linearization of the augmented Lagrangian in separable or nonseparable, smooth or nonsmooth nonlinear convex programming. Thanks to this parallel decomposable property, excellent numerical performance can be achieved. (see parallel computing software such as DistOpt~\cite{Distopt1,Distopt2})
\subsection{Our previous work on NCCP and motivation of further study}\label{1.2}
\indent There are two types of NCCP problems mentioned by Cohen and Zhu~\cite{CohenZ} as follows:
\begin{equation*}\label{Prob:general-function-al}
\begin{array}{|l|l|}
\hline
\mbox{NCCP with nonsmooth constraints}&\mbox{NCCP with smooth constraints}\\
\hline
\begin{array}{lll}
\mbox{(P$_{1}$):}&\min     &G(u)+J(u)\\
                 &\rm {s.t}&\Theta(u)=\Phi(u)\in -\mathbf{C}\\
                 &         & u\in U.\\
\end{array}
&
\begin{array}{lll}
\mbox{(P$_{2}$):}&\min      &G(u)+J(u)      \\
                 &\rm {s.t} &\Theta(u)=\Omega(u)\in -\mathbf{C} \\
                 &          & u\in U.\\
\end{array}\\
\hline
\end{array}
\end{equation*}
These two problems could be seen as special cases of NCCP. Cohen and Zhu~\cite{CohenZ} proposed the APP-AL to solve (P$_1$):\\
\noindent\rule[0.25\baselineskip]{\textwidth}{1.5pt}
{\bf Auxiliary Problem Principle (APP-AL) for solving (P$_1$): Algorithm 14 in~\cite{CohenZ}}\\
\noindent\rule[0.25\baselineskip]{\textwidth}{0.5pt}
{Initialize} $u^0 \in \mathbf{U}$ and $p^0\in \mathbf{C^*}$  \\
 \textbf{for} $k = 0,1,\cdots $, \textbf{do}
\begin{eqnarray}
\mbox{(AP$^{k}$)}\quad u^{k+1}&\leftarrow&\min_{u\in \mathbf{U}}\langle\nabla G(u^{k}), u \rangle + J(u)+ \langle\Pi(p^k+\gamma\Phi(u^k)), \Phi(u)\rangle+\frac{1}{\epsilon}D(u,u^k);\label{primal}\\
p^{k+1}&\leftarrow&p^k+\frac{\rho}{\gamma}\bigg{[}\Pi\big{(}p^k+\gamma\Phi(u^{k+1})\big{)}-p^k\bigg{]}.\qquad\quad\qquad\qquad\qquad\qquad\qquad\qquad\label{dual}
\end{eqnarray}
\textbf{end for}\\
\noindent\rule[0.25\baselineskip]{\textwidth}{1.5pt}
\indent In the APP-AL algorithm, a core function $K(u)$ is introduced. The objective function of (AP$^k$) is obtained by keeping the nonsmooth part $J(u)$ and $\Phi(u)$, linearizing the smooth part $G(u)$ and the nonlinear term $\varphi(\Phi(u),p)=[\|\Pi\big{(}p+\gamma\Phi(u)\big{)}\|^2-\|p\|^2]/2\gamma$ in the augmented Lagrangian, and adding a regularization term $\frac{1}{\epsilon}D(u,u^k)=\frac{1}{\epsilon}[K(u)-K(u^k)-\langle\nabla K(u^k),u\rangle]$ (Bregman distance function). $\Pi(\cdot)$ is the projection on $\mathbf{C}^*$. In~\cite{CohenZ}, it is shown that the sequence generated by this algorithm converges to the saddle point of (P$_1$).\\
\indent To solve (P$_2$) with smooth nonseparable mapping $\Omega(u)$, they
also proposed a variant algorithm in which the term involving $\Omega(u)$ in (AP$^k$) is replaced by $\langle\Pi\big{(}p^k+\gamma\Omega(u^k)\big{)},\nabla\Omega(u^k)\cdot u\rangle$, but the formal convergence analysis is not given.\\
\indent Regarding decomposition, the interesting part of the APP-AL algorithm is as follows. Assume the following space decomposition of $\mathbf{U}$:\\
\begin{equation}\label{space_decomposition}
\mathbf{U}=\mathbf{U}_1\times\mathbf{U}_2\cdots\times\mathbf{U}_N, \mathbf{U}_i\subset \RR^{n_i}, \sum_{i=1}^{N}n_i=n.
\end{equation}
\indent For the structured problem (P$_1$), where $J(u)=\sum_{i=1}^{N}J_i(u_i)$ and $\Phi(u)=\sum_{i=1}^{N}\Phi_i(u_i)$, if we chose an additive core function $K(u)=\sum_{i=1}^{N}K_i(u_i)$,
then the problem (AP$^k$) splits into $N$ independent subproblems. Additionally, APP-AL has wide applications in engineering systems. In particular, this approach was adopted by Kim and Baldick and by Renaud to parallelize optimal power flow in very large interconnected power systems~\cite{Electro97,KimBaldick00,Electro93}. For effective implementation of APP-AL, choice of parameters is the key factor affecting the convergence performance of the algorithm. (Cao et al.~\cite{APPturning1}, Hur et al.~\cite{APPturning2})\\
\indent Large-scale optimization has recently attracted significant attention due to its important role in big data analysis. Applications found in various areas have drawn renewed attention to research on the convergence rate analysis. In this paper we further investigate APP-AL and propose a new algorithm to solve NCCP. Specifically, we focus on the following issues:
\begin{itemize}
\item[(i)] Propose a flexible Variant Auxiliary Problem Principle (VAPP) algorithm for solving NCCP problems.
\item[(ii)] Derive better convergence rates of the VAPP algorithm to solve general convex and strongly convex problem (P).
\item[(iii)] Study error bound conditions to ensure the linear convergence of the VAPP algorithm, and derive some verifiable sufficient condition for error bound property.
\item[(iv)] Investigate the Forward-Backward Splitting (FBS) formulation for the VAPP algorithm, and establish the connection between VAPP algorithm and other primal-dual splitting methods.
\item[(v)] For practical reasons, propose some technique to overcome the difficulty in the implementation of the VAPP algorithm, including the backtracking strategy, estimate the dual bound, and explore $\mathbf{C}$-convexity of structured mapping to some special cones.
\end{itemize}
\subsection{Related work}\label{1.3}
\indent In recent years, the research on decomposition method for nonlinear optimization with constraints can be classified four lines: Alternate method of augmented Lagrangian, partial linearization of augmented Lagrangian, saddle point method, and splitting method.\\
\indent First we review some ADMM-type schemes. The celebrated ADMM traces back to the work of Fortin and Glowinski~\cite{ADMM1983}, and Gabay and Mercier~\cite{GabayMercier1976}.~\cite{Heyuan,Monteiro2013,LiSunToh2016,GaoZhang2017,Aybat2016} establish the worse-case $O(1/t)$ sub-linear convergence rate of ADMM and its extension. For convex minimization model with linear constraints, the global linear convergence rate of ADMM is proved in~\cite{DengYin,HongLuo2017,LinMaZhang2015,LiuYuanZengZhang2018}.\\
\indent Secondly, we review some works based on partial linearization of augmented Lagrangian and proximal like iterations. APP-AL (Cohen and Zhu~\cite{CohenZ}) is described in Subsection~\ref{1.2}. Another important work is the predictor corrector proximal multiplier method (PCPM) proposed by Chen and Teboulle~\cite{ChenTeboulle1994}. Their inexact method allows for computing the primal steps approximately, the convergence is provided under a mild assumption. Linear convergence is provided whenever the inverse of KKT mapping is Lipschitz continuous at the origin. Later, Zhang et al.~\cite{Zhang2011} introduced a unified primal dual method for nonlinear convex optimization with linear constraints. The general idea of their method is to replace the augmented Lagrangian minimization by proximal-like iterations in the Uzawa algorithm.\\
\indent Next we present some work on the saddle point method. Chambolle and Pock~\cite{ChambollePock11,ChambollePock16} proposed a primal-dual algorithm (PDA) that can solve convex-concave saddle point problem: $\min\limits_{x}\max\limits_{y}f(x)-g(y)+\langle Kx,y\rangle$. This method can be interpreted as a preconditioned ADMM. The sequence generated by PDA converges to one saddle point with $O(1/t)$ ergodic convergence rate. $O(1/t^2)$ rate and linear convergence are also proposed in their work. For nonlinear convex-concave saddle point problem: $\min\limits_{x}\max\limits_{y}\phi(x,y)$, Nemirovski et. al.~\cite{Nemirovski2004} proposed a Mirror-Prox algorithm that can solve it with $O(1/t)$ rate. For the strongly concave case, they also proposed the $O(1/t^2)$ rate of Mirror-Prox~\cite{Mirror-Prox-1,Mirror-Prox-2}. Recently, Hamedani and Aybat~\cite{Aybat-nonlinear} proposed a PDA that can solve a more complex convex-concave saddle point problem: $\min\limits_{x}\max\limits_{y}f(x)-g(y)+\phi(x,y)$. They showed global convergence and provided ergordic iteration complexity $O(1/t)$ in terms of the primal-dual gap function. $O(1/t^2)$ rate is also proposed for the case $f$ is strongly convex.\\
\indent  Finally, we review the works on splitting. As stated in~\cite{Combettes2018}, many different primal-dual splitting algorithm are explicitly or implicitly, reformulations of three basic schemes: Forward-Backward Splitting (FBS)~\cite{Mercier1979}, Douglas-Rachford Splitting (DRS)~\cite{LionsMercier1979} and Tseng's Forward-Backward-Forward Splitting (FBFS)~\cite{Tseng2000}.\\
\indent Various primal-dual splitting methods are used to solve the composite optimization problem:
\begin{equation}\label{prob:11}
\min_{u} f(Au)+g(u),\quad A\in\RR^{m\times n}
\end{equation}
which can be reformulated as the equality constrained problem
\begin{equation}\label{prob:22}
\begin{array}{cc}
\min\limits_{u,v} &f(v)+g(u)\\
{\rm s.t.} &Au-v=0
\end{array}
\end{equation}
\indent In~\cite{ConnorVandenberghe2014}, O'Connor and Vandenberghe discuss some primal-dual splitting methods for solving this problem. They indicate that ADMM, Spingarn's method of partial inverses and Chambolle-Pock method may be rendered by DRS. Recently,~\cite{ConnorVandenberghe2017} showed the equivalence of the primal-dual hybrid gradient method (PDHG) and DRS. Esser et al.~\cite{Esser2010} proposed a generalized PDHG algorithm and other proximal FBS methods for solving problem~\eqref{prob:22}, its dual problem and saddle point formulation problem. Tseng proposed the FBFS method to solve the inclusion problem and provide the convergence of this method. His work is motivated by the extra-gradient method for monotone variational inequality. Compared with FBS method, FBFS needs an additional forward step and projection onto set $\mathbf{X}$. Furthermore, if the inverse of mapping is local Lipschitz, then his method has a local linear rate of convergence.
\subsection{Contributions and organization of this paper}
\indent In this paper, we generalize APP-AL~\cite{CohenZ} to the VAPP method for solving NCCP where the objective function and constraints are smooth and may be nonsmooth. Each iteration of VAPP generates a nonlinear approximation to the primal problem of an augmented Lagrangian method. The approximation incorporates both linearization and a variable distance-like function, then the iterations of VAPP provide one decomposition property for NCCP. The main contributions of this work are the following.
\begin{itemize}
\item[{\rm(i)}] We propose an error bound based on VAPP's iterations, and linear convergence under this condition is provided. We also derive a verifiable sufficient condition for this error bound.
\item[{\rm(ii)}] For the general convex case (without error bound condition), we establish $O(1/t)$ convergence rate results for primal suboptimality, feasibility and dual suboptimality. By adaptively setting in parameters at different iteration, we show $O(1/t^2)$ convergence rate for the strongly convex case.
\item[\rm{(iii)}] In addition, we propose the Forward-Backward splitting formulation of VAPP method and establish the connection between VAPP and other primal-dual splitting methods.
\end{itemize}
\indent Finally, we propose some techniques to overcome the difficulty in implementation of the VAPP method.\\
\indent The rest of this paper is organized as follows. Section~\ref{pre} is devoted to the preliminaries that we will use in this paper. In Section~\ref{VAPP-a}, we propose the updating scheme VAPP for solving NCCP problems. Convergence and convergence rate analyses are also provided. Additionally, we propose the $O(1/t^2)$ convergence rate for strongly convex case. In Section~\ref{sec:linear_convergence}, we provide the linear convergence of VAPP with various error bounds. Section~\ref{sec:FBS-relation} describes an FBS formulation for VAPP methods and explains the connection with other primal-dual splitting methods. In the Section~\ref{sec:implimentation}, we further study a variant VAPP with different assumption and the issues in the implementation of VAPP for NCCP. Finally, Section~\ref{sec:numerical} presents numerical experiments for Ivanov-type structured elastic net-SVM problem.
\section{Preliminaries}\label{pre}
\indent In this section, we recall the notation for the Lagrangian and augmented Lagrangian for nonlinear optimization with cone constraints and the projection onto a convex set.
\subsection{Lagrangian and augmented Lagrangian duality and saddle point optimality conditions for nonlinear cone optimization}
\indent The original Lagrangian of problem (P) is $L(u,p)=(G+J)(u)+\langle p,\Theta(u)\rangle$,
and a saddle point $(u^*,p^*)\in \mathbf{U}\times \mathbf{C}^*$ is a point such that
\begin{equation}
\forall u\in \mathbf{U},\; \forall p\in \mathbf{C}^{*}:\; L(u^{*},p)\leq L(u^{*},p^{*})\leq L(u,p^{*}). \label{saddle point:L}
\end{equation}
\indent The dual function $\psi$ is defined as $\psi(p)=\min_{u\in \mathbf{U}}L(u,p)$, $\forall p\in\mathbf{C}^*$, which is concave and sub-differentiable. We consider the primal-dual pair of nonlinear convex cone optimization problems:
\begin{equation*}
\begin{array}{lllllll}
\mbox{(P):} &\min      & (G+J)(u)                                   &\qquad\qquad&\mbox{(D):}    &\max      & \psi(p) \\
                &\rm {s.t} & \Theta(u)=\Omega(u)+\Phi(u)\in -\mathbf{C} &\qquad\qquad&               &\rm {s.t} & p\in\mathbf{C}^*.\\
                &          & u\in\mathbf{U}                             &\qquad\qquad&               &          &
\end{array}
\end{equation*}
\indent Throughout this paper, we make the following standard assumptions for problem (P):
\begin{assumption}\label{assump1}
{\rm(H$_1$)} $J$ is a convex, l.s.c. function (not necessarily differentiable) such that $\mathbf{dom}J\cap \mathbf{U}\neq \emptyset$.\\
{\rm(H$_2$)} $G$ is convex and differentiable; its derivative is Lipschitz with constant $B_{G}$.\\
{\rm(H$_3$)} $\Omega$ is $\mathbf{C}$-convex mapping from $\mathbf{U}$ to $\mathbf{C}$, where
\begin{equation}\label{Theta_C_convex}
\forall u,v\in \mathbf{U}, \forall \alpha\in [0,1], \Omega(\alpha u+(1-\alpha)v)-\alpha\Omega(u)-(1-\alpha)\Omega(v)\in - \mathbf{C}.
\end{equation}
$\Phi$ is also $\mathbf{C}$-convex mapping from $\mathbf{U}$ to $\mathbf{C}$.\\
{\rm(H$_4$)}  $\Omega$ is differentiable, the derivative of function $f_p(u)=\langle p,\Omega(u)\rangle$ is Lipschitz on $\mathbf{U}$ with constant $B_{\Omega}$ uniformly in $p\in\RR^m$, such that
\begin{equation}\label{eq:AssumpH4}
\forall u,v\in \mathbf{U}, \|\nabla f_p(u)-\nabla f_p(v)\|\leq B_{\Omega}\|u-v\|.
\end{equation}
{\rm(H$_5$)} $\Theta(u)$ is Lipschitz with constant $\tau$ on an open subset $\mathcal{O}$ containing $\mathbf{U}$, where
\begin{equation}\label{Theta_Lipschitz}
\forall u,v\in \mathcal{O}, \|\Theta(u)-\Theta(v)\|\leq\tau\|u-v\|.
\end{equation}
{\rm(H$_6$)} Constraint Qualification Condition. When $\mathring{\mathbf{C}}\neq\emptyset$, we assume that
\begin{equation}\label{Theta_CQC_neq}
\mbox{{\bf CQC:}}\qquad\qquad\qquad\qquad\quad\Theta(\mathbf{U})\cap(-\mathring{\mathbf{C}})\neq\emptyset.\qquad\qquad\qquad\qquad\qquad\quad
\end{equation}
For the case $\mathbf{C}=\{0\}$, we assume that
$0\in \mbox{interior of}\quad\Theta(\mathbf{U})$.\\
{\rm(H$_7$)} There exists at least one saddle point for Lagrangian of {\rm(P)}.
\end{assumption}
\indent Conditions (H$_1$)-(H$_3$) guarantee that (P) is a convex problem. The CQC condition (H$_6$) implies that the Lagrangian dual function is coercive and that the dual optimal solution set is bounded~\cite{CohenZ}. \\
\indent Under Assumption~\ref{assump1}, by Theorem 3.2.12 of~\cite{Ortega70}, for any $p\in\RR^m$, the following descent property of $G$ and $\langle p,\Omega\rangle$ holds:
\begin{eqnarray}
G(v)-G(u)-\langle\nabla G(u),v-u\rangle&\leq&\frac{B_G}{2}\|u-v\|^2,\label{eq:Lip_G}\\
\langle p,\Omega(v)-\Omega(u)-\nabla\Omega(u)(v-u)\rangle&\leq&\frac{B_{\Omega}}{2}\|u-v\|^2.\label{eq:Lip_Omega}
\end{eqnarray}
\indent For convex problem (P), the primal-dual pair $(u^*,p^*)$ is a saddle point if and only if $u^*$ and $p^*$ are optimal solutions to the primal and dual problems (P) and (D), respectively, with no duality gap, that is, $(G+J)(u^*)=\psi(p^*)$. (See Shapiro and Scheinberg~\cite{Shapiro2000})\\
\indent It is well known that augmented Lagrangians are a remedy to the duality gaps encountered with original Lagrangians for nonconvex problems. As we shall see, augmented Lagrangians are also useful for convex, but not strongly convex, problems.\\
\indent The augmented Lagrangian associated with problem (P) is defined as
\begin{equation}\label{func:AL_with_xi}
L_{\gamma}(u,p)=\min_{\xi\in-\mathbf{C}}(G+J)(u)+\langle p,\Theta(u)-\xi\rangle+\frac{\gamma}{2}\|\Theta(u)-\xi\|^{2}.
\end{equation}
Consider the following function $\varphi: \RR^m\times\RR^n\rightarrow \RR$:
\begin{equation}\label{func:varphi_1}
\varphi(\theta,p)=\min_{\xi\in-\mathbf{C}}\langle p,\theta-\xi\rangle + \frac{\gamma}{2}\|\theta-\xi\|^{2}.
\end{equation}
Introducing a multiplier $q\in \mathbf{C}^*$ for the minimization problem~\eqref{func:varphi_1} with respect to the linear cone constraint, we obtain the equivalent formulation for $\varphi(\theta,p)$:
\begin{eqnarray}\label{func:varphi_2}
\varphi(\theta,p)&=&\max_{q\in\mathbf{C}^*}\min_{\xi}\langle p, \theta-\xi\rangle+\frac{\gamma}{2}\|\theta-\xi\|^2+\langle q,\xi\rangle\nonumber\\
                 &=&\max_{q\in\mathbf{C}^*}\langle q,\theta\rangle-\frac{\|q-p\|^2}{2\gamma}.
\end{eqnarray}
This provides the explicit expression $L_\gamma(u,p)=(G+J)(u)+\varphi(\Theta(u),p)$, with $\varphi(\Theta(u),p)=[\|\Pi\big{(}p+\gamma\Theta(u)\big{)}\|^2-\|p\|^2]/2\gamma$. The augmented Lagrangian dual function is defined as:
\begin{equation}
\forall p\in \RR^m, \psi_\gamma(p)=\min_{u\in\mathbf{U}}L_\gamma(u,p)=\min_{u\in\mathbf{U}}(G+J)(u)+\varphi(\Theta(u),p).
\end{equation}
Using $\psi_\gamma(p)$, we obtain the following new primal-dual pair of nonlinear convex cone optimization problems:
\begin{equation*}
\begin{array}{lllllll}
  \mbox{(P):} &\min      & (G+J)(u)                 &\qquad\qquad\qquad\qquad&\mbox{(D$_\gamma$):}    &\max      & \psi_\gamma(p) \\
              &\rm {s.t} & \Theta(u)\in -\mathbf{C} &\qquad\qquad\qquad\qquad&                      &\rm {s.t} & p\in\mathbf{R}^m\\
              &          & u\in\mathbf{U}           &\qquad\qquad\qquad\qquad&                      &          &
\end{array}
\end{equation*}
The saddle point of the augmented Lagrangian $(u^*,p^*)\in\mathbf{U}\times\mathbf{R}^m$ is defined as
\begin{equation}
\forall u\in \mathbf{U},\; \forall p\in \mathbf{R}^{m}:\; L_\gamma(u^{*},p)\leq L_\gamma(u^{*},p^{*})\leq L_\gamma(u,p^{*}). \label{saddle point:AL}
\end{equation}
The authors of \cite{CohenZ} show that $L$ and $L_\gamma$ have the same sets of saddle points $\mathbf{U}^*\times\mathbf{P}^*$ on $\mathbf{U}\times\mathbf{C}^*$ and $\mathbf{U}\times\mathbf{R}^m$, respectively. The point $(u^*,p^*)$ is a saddle point if and only if $u^*$ and $p^*$ are optimal solutions to the primal and dual problems (P) and (D$_\gamma$), respectively.
\subsection{The properties of projection on convex set}
\indent Let $\mathcal{S}$ be a nonempty closed convex set of $\RR^m$. For $u\in\RR^m$, let $\Pi_{\mathcal{S}}(u)$ be the projection on $\mathcal{S}$. Then we have that~\cite{Projectiononconvexsets}:
\begin{eqnarray}
&(i)&\langle v-\Pi_{\mathcal{S}}(u), u-\Pi_{\mathcal{S}}(u)\rangle\leq0, \forall v\in\mathcal{S};\label{eq:Projecproperty3}\\
&(ii) &\|\Pi_{\mathcal{S}}(u)-\Pi_{\mathcal{S}}(v)\|\leq\|u-v\|, \forall v\in\RR^m.\label{eq:Projecproperty4}
\end{eqnarray}
Another useful property of the projection operator is given by the following proposition.
\begin{proposition}\label{proposition}
\noindent For any $(u,v,w)\in\RR^{m\times m\times m}$, the projection operator $\Pi_{\mathcal{S}}$ satisfies
\begin{equation}\label{eq:Projecproperty7}
2\langle\Pi_{\mathcal{S}}(w+u)-\Pi_{\mathcal{S}}(w+v),u\rangle\leq\|u-v\|^2+\|\Pi_{\mathcal{S}}(w+u)-w\|^2-\|\Pi_{\mathcal{S}}(w+v)-w\|^2.
\end{equation}
\end{proposition}
\proof
Since $\Pi_{\mathcal{S}}(w+u)\in\mathcal{S}$, using the property of projection~\eqref{eq:Projecproperty3}, we have that
\begin{eqnarray*}
\langle\Pi_{\mathcal{S}}(w+u)-\Pi_{\mathcal{S}}(w+v),w+v-\Pi_{\mathcal{S}}(w+v)\rangle\leq 0.
\end{eqnarray*}
Then we have that
\begin{eqnarray*}
2\langle\Pi_{\mathcal{S}}(w+u)-\Pi_{\mathcal{S}}(w+v),v\rangle &\leq& 2\langle\Pi_{\mathcal{S}}(w+u)-\Pi_{\mathcal{S}}(w+v),\Pi_{\mathcal{S}}(w+v)-w\rangle\\
                                    &=&\|\Pi_{\mathcal{S}}(w+u)-w\|^2-\|\Pi_{\mathcal{S}}(w+u)-\Pi_{\mathcal{S}}(w+v)\|^2-\|\Pi_{\mathcal{S}}(w+v)-w\|^2.
\end{eqnarray*}
It is clear that
$$2\langle\Pi_{\mathcal{S}}(w+u)-\Pi_{\mathcal{S}}(w+v),u-v\rangle\leq\|u-v\|^2+\|\Pi_{\mathcal{S}}(w+u)-\Pi_{\mathcal{S}}(w+v)\|^2.$$
Adding the preceding two inequalities, we have~\eqref{eq:Projecproperty7}.\\
\qed
Next, we consider the projection onto a convex cone. Let $\Pi$ and $\Pi_{-\mathbf{C}}$ be the projection on $\mathbf{C}^*$ and $-\mathbf{C}$. The projection is characterized by the following conditions (see Wierzbicki~\cite{Projection}):
\begin{eqnarray}
&(iii)  &v=\Pi(v)+\Pi_{-\mathbf{C}}(v), \forall v\in\mathbf{R}^m;\label{eq:Projecproperty5}\\
&(iv) &\langle\Pi(v), \Pi_{-\mathbf{C}}(v)\rangle=0, \forall v\in\mathbf{R}^m.\label{eq:Projecproperty6}
\end{eqnarray}
\section{VAPP method for solving NCCP}\label{VAPP-a}
\subsection{Scheme VAPP and solutions for primal subproblem}
Based on the augmented Lagrangian theory, in this subsection we will establish a new first-order primal-dual augmented Lagrangian algorithm to solve (P). We introduce the core function $K(\cdot)$ and variable parameter $\epsilon^k$, $\epsilon^k>0$. $K(\cdot)$ satisfies the following assumption:
\begin{assumption}\label{assump2}
$K$ is strong convex with parameter $\beta>0$ and differentiable with its gradient Lipschitz continuous with the parameter $B$ on $\mathbf{U}$.
\end{assumption}
\indent Note that $D(u,v)=K(u)-K(v)-\langle\nabla K(v),u-v\rangle$ is a Bregman-like function~\cite{Mirror,CohenZ}. From Assumption~\ref{assump2} we have that $\frac{\beta}{2}\|u-v\|^2\leq D(u,v)\leq\frac{B}{2}\|u-v\|^2$.\\
\indent We assume the sequence $\{\epsilon^k\}$ satisfies:
\begin{eqnarray}\label{eq:parameter}
0<\underline{\epsilon}\leq\epsilon^{k+1}\leq\epsilon^{k}\leq\overline{\epsilon}<\beta/\big{(}B_{G}+B_{\Omega}+\gamma\tau^{2}\big{)}.\qquad\qquad\label{parameter-choice-a}
\end{eqnarray}
\indent For given $u^k$ and $p^k$, we take following approximation of augmented Lagrangian $L_\gamma(u,p)=(G+J)(u)+\varphi(\Theta(u),p)$:
\begin{eqnarray*}
\tilde{L}_\gamma(u,p)&=&G(u^k)+\langle \nabla G(u^k),u-u^k\rangle+J(u)+\varphi\big{(}\Theta(u^k),p^k\big{)}\nonumber\\
&&+\langle\Pi\big{(}p^k+\gamma\Theta(u^k)\big{)},\nabla\Omega(u^k)(u-u^k)+\Phi(u)-\Phi(u^k)\rangle
+\frac{1}{\epsilon^k}D(u,u^k),
\end{eqnarray*}
where $\Pi\big{(}p^k+\gamma\Theta(u^k)\big{)}=\nabla_\theta\varphi\big{(}\Theta(u^k),p^k\big{)}$. Based on the above approximation of augmented Lagrangian $L_\gamma(u,p)=(G+J)(u)+\varphi(\Theta(u),p)$, we propose the following first-order primal-dual method for solving the NCCP problem (P):\\
\noindent\rule[0.25\baselineskip]{\textwidth}{1.5pt}
{\bf VAPP: Variant Auxiliary Problem Principle for solving (P)}\\
\noindent\rule[0.25\baselineskip]{\textwidth}{0.5pt}
{Initialize} $u^0 \in \mathbf{U}$ and $p^0\in \mathbf{C^*}$  \\
 \textbf{for} $k = 0,1,\cdots $, \textbf{do}
\begin{eqnarray}
u^{k+1}&\leftarrow&\min_{u\in \mathbf{U}}\langle\nabla G(u^{k}), u \rangle + J(u)+ \langle q^k, \nabla\Omega(u^k)u+\Phi(u)\rangle+\frac{1}{\epsilon^k}D(u,u^k);\label{primal}\\
p^{k+1}&\leftarrow&\Pi\big{(}p^k+\rho\Theta(u^{k+1})\big{)}.\quad\qquad\qquad\qquad\qquad\qquad\qquad\qquad\qquad\qquad\qquad\label{dual}
\end{eqnarray}
\textbf{end for}\\
\noindent\rule[0.25\baselineskip]{\textwidth}{1.5pt}
where $q^k=\Pi\left(p^k+\rho\Theta(u^k)\right)$. Additionally, for simplicity of computation, we select $\rho=\gamma$. Assume the space decomposition~\eqref{space_decomposition} of $\mathbf{U}$, to solve problem (P) with $J(u)=\sum\limits_{i=1}^{N}J_i(u_i)$ and $\Phi(u)=\sum\limits_{i=1}^{N}\Phi_i(u_i)$, VAPP keeps the parallel decomposition property of APP-AL. Furthermore, if $J_i(u_i)$ and $\Phi_i(u_i)$ are quadratic or $\ell_{\nu}$ norms, $\nu=\{1,2,\infty\}$, then "$u$ update" in VAPP has a closed-form for each coordinate $u_i$.
\subsection{Convergence and convergence rate analysis of VAPP for convex problem (P)}
\indent Before proceeding convergence analysis of VAPP, we first give the generalized equilibrium reformulation for saddle point inequality~\eqref{saddle point:L}:\\
\indent Find $(u^*, p^*)\in\mathbf{U}\times\mathbf{C}^*$ such that
\begin{equation}\label{VIS}
\mbox{(EP):}\qquad L(u^*, p)-L(u, p^*)\leq 0, \forall u\in\mathbf{U}, p\in\mathbf{C}^*.
\end{equation}
Obviously, for given $u\in\mathbf{U}$, $p\in\mathbf{C}^*$, bifunction $L(u',p)-L(u,p')$ is convex in $u'$ and linear in $p'$. For $u,v\in\mathbf{U}$, define
\begin{eqnarray}\label{eq:delta}
\Delta^k(u,v)&=&D(v,u)-\epsilon^k\bigg{[}\big{(}G(v)-G(u)-\langle\nabla G(u),v-u\rangle\big{)}\nonumber\\
&&+\langle q^k, \Omega(v)-\Omega(u)-\nabla\Omega(u)(v-u)\rangle+\frac{\gamma}{2}\|\Theta(u)-\Theta(v)\|^2\bigg{]}.\;\label{definitiondelta}
\end{eqnarray}
By Assumptions~\ref{assump1},~\ref{assump2},~\eqref{eq:Lip_G} and~\eqref{eq:Lip_Omega}, obviously, we have that
\begin{eqnarray}\label{eq:bounddelta}
\Delta^k(u,v)&\geq&\frac{\beta-\epsilon^k(B_G+B_{\Omega}+\gamma\tau^2)}{2}\|u-v\|^2.\label{bounddelta}
\end{eqnarray}
\indent For $u\neq v$, if the term $\Delta^k(u,v)$ is negative, then the satisfication constraint $\epsilon^k<\frac{\beta}{B_G+B_{\Omega}+\gamma\tau^2}$ falls. This fact follows the backtracking strategy of VAPP (see section~\ref{implementation})
\indent The following lemma gives the descent property for generalized distance $D(u,u')+\frac{\epsilon^k}{2\gamma}\|p-p'\|^2$.
\begin{lemma}\label{lemma:bound1} {\bf(Descent inequalities of generalized distance function)}\\
Suppose Assumptions~\ref{assump1} and~\ref{assump2} hold, $\{(u^k,p^k)\}$ is generated by VAPP, and the parameter sequence $\{\epsilon^k\}$ satisfies~\eqref{eq:parameter}. Then for any $u\in\mathbf{U}$, $p\in\mathbf{C}^*$, $k\in\mathbb{N}$ descent property of generalized distance function holds
\begin{eqnarray*}
&&\big{[}D(u,u^{k+1})+\frac{\epsilon^{k+1}}{2\gamma}\|p-p^{k+1}\|^2\big{]}-\big{[}D(u,u^k)+\frac{\epsilon^k}{2\gamma}\|p-p^{k}\|^2\big{]}\\
&\leq&\epsilon^k[L(u,q^k)-L(u^{k+1},p)]-\big{[}\Delta^k(u^k,u^{k+1})+\frac{\epsilon^k}{2\gamma}\|q^k-p^k\|^2\big{]}
\end{eqnarray*}
\end{lemma}
\begin{proof}
See Appendix {\bf A$_1$}.\qed
\end{proof}
\indent Now we are ready to prove the convergence of VAPP.
\begin{theorem}\label{theo:general-convergence} {\bf(Convergence analysis for VAPP)}\\
Suppose Assumption \ref{assump1} and Assumption~\ref{assump2} hold, and the sequence $\{\epsilon^k\}$ satisfies~\eqref{eq:parameter}. Let $(u^*,p^*)$ be a saddle point of $L$ over $\mathbf{U}\times\mathbf{C}^*$. Then the sequence $\{(u^{k},p^{k})\}$ generated by VAPP is bounded and converges to $(u^{*},p^{*})$.
\end{theorem}
\begin{proof} See Appendix {\bf A$_2$}. \qed
\end{proof}
Next we analyze the convergence rate of VAPP. For any integer number $t$, let $\bar{u}_{t}=\frac{\sum_{k=0}^{t}\epsilon^ku^{k+1}}{\sum_{k=0}^{t}\epsilon^k}$ and $\bar{p}_{t}=\frac{\sum_{k=0}^{t}\epsilon^kq^{k}}{\sum_{k=0}^{t}\epsilon^k}$. For the case where $\epsilon^k=\epsilon$, one construct average point $\bar{u}_{t}=\frac{\sum_{k=0}^{t}u^{k+1}}{t+1}$ and $\bar{p}_{t}=\frac{\sum_{k=0}^{t}q^{k}}{t+1}$. The following theorem shows $\bar{u}_{t}$ is one approximation solution of (P) with $O(1/t)$, thus proving a convergence rate of $O(1/t)$ in the worst case for the VAPP algorithm.

\begin{theorem}\label{thm:ergodic_iteration_complexity} {\bf(Bifunction value estimation, primal suboptimality and feasibility for solving (P) by VAPP)}\\
\noindent Suppose Assumptions~\ref{assump1} and~\ref{assump2} hold, let $(u^*,p^*)$ be a saddle point, $M_0$ be a bound of dual optimal solution of (P), the parameter sequence $\{\epsilon^k\}$ satisfy~\eqref{eq:parameter}, and for any integer number $t>0$, we have $(\bar{u}_{t},\bar{p}_{t})\in\mathbf{U}\times\mathbf{C}^*$ and:\\
\begin{itemize}
\item[{\rm(i)}] Global estimate in bifunction values of (EP):
\begin{eqnarray*}
L(\bar{u}_{t},p)-L(u,\bar{p}_{t})\leq
\frac{D(u,u^0)+\frac{\epsilon^0}{2\gamma}\|p-p^{0}\|^2}{\underline{\epsilon}(t+1)},\;\forall (u,p)\in\mathbf{U}\times\mathbf{C}^*.
\end{eqnarray*}
\item[{\rm(ii)}] Feasibility:
$$\|\Pi\big{(}\Theta(\bar{u}_{t})\big{)}\|\leq\frac{d_1}{\underline{\epsilon}(t+1)},$$
where $d_1=\max\limits_{\|p\|\leq M_0+1}\big{[}D(u^*,u^0)+\frac{\epsilon^0}{2\gamma}\|p-p^{0}\|^2\big{]}$.
\item[{\rm(iii)}] Primal suboptimality:
$$-\frac{M_0d_1}{\underline{\epsilon}(t+1)}\leq(G+J)(\bar{u}_{t})-(G+J)(u^*)\leq \frac{d_1}{\underline{\epsilon}(t+1)}.$$
\end{itemize}
\end{theorem}
\begin{proof} See Appendix {\bf A$_3$}.\qed
\end{proof}
Observe that Theorem~\ref{thm:ergodic_iteration_complexity} prompts VAPP to have the convergence rate $O(1/t)$ in the worst case. To obtain the dual suboptimality, we need the following additional assumption.
\begin{assumption}\label{assump3}
$G+J$ is coercive on $\mathbf{U}$, if $\mathbf{U}$ is not bounded, that is,
\begin{eqnarray*}
\forall \{u^{k}|k\in \mathbb{N}\}\subset \mathbf{U}, \lim_{k\rightarrow+\infty}\|u^{k}\|=+\infty\Rightarrow\lim_{k\rightarrow+\infty}(G+J)(u^{k})=+\infty.
\end{eqnarray*}
\end{assumption}
\indent The following lemma states that for any given bounded set of dual points, the corresponding optimizer of the augmented Lagrangian is bounded.
\begin{lemma}\label{lemma:ALBounded}
Suppose Assumptions~\ref{assump1} and~\ref{assump3} hold. Then we have a positive constant $d_u$, for any $p\in\RR^m$ and $\|p\|\leq d_p$, there is an optimizer $\hat{u}(p)\in\arg\min\limits_{u\in\mathbf{U}}L_{\gamma}(u,p)$ such that $\|\hat{u}(p)\|\leq d_u$.
\end{lemma}
\proof See Appendix {\bf A$_4$}.\qed
\indent From Theorem~\ref{theo:general-convergence}, the sequence $\{(u^k,p^k)\}$ is bounded; therefore there exist positive number $\mu$ such that for all $k\in\mathbb{N}$, $\|u^k\|\leq\mu$ and $\|p^k\|\leq\mu$. Obviously we also have that $\|\bar{u}\|\leq\mu$ and $\|\bar{p}\|\leq\mu$. Moreover, we have that
\begin{eqnarray*}
\|q^k\|&\leq&\|q^k-p^{k+1}\|+\|p^{k+1}\|\leq\gamma\tau\|u^k-u^{k+1}\|+\|p^{k+1}\|\\
       &\leq&\gamma\tau(\|u^k\|+\|u^{k+1}\|)+\|p^{k+1}\|\leq(1+2\gamma\tau)\mu.
\end{eqnarray*}
Denote $\mathfrak{B}^{p}=\big{\{}p|\|p\|\leq r^p\big{\}}$ with $r^p=(1+2\gamma\tau)\mu$. Therefore, $p^k,\bar{p},q^k\in\mathfrak{B}^{p}$, $\forall k\in\mathbb{N}$. Furthermore, from Lemma~\ref{lemma:ALBounded} for $p\in\mathfrak{B}^{p}$, we have that $\hat{u}(p)\in\arg\min L_\gamma(u,p)$ and $\|\hat{u}(p)\|\leq d_u$. Specifically, we construct a ball as follows: $\mathfrak{B}^{u}=\{u|\|u\|\leq r^u \}$ with $r^u=\max(\mu,d_u)$. Then, $u^k\in\mathfrak{B}^{u}$ and $\hat{u}(p)\in\mathfrak{B}^{u}$ for every $p\in\mathfrak{B}^{p}$.\\
\indent The next theorem provides the convergence rate for approximate saddle point and dual suboptimality for VAPP.
\begin{theorem}\label{theo_2}{\bf(Approximate saddle point and dual suboptimality for solving (P) by VAPP)}\\
Suppose Assumption~\ref{assump1},~\ref{assump2} and~\ref{assump3} hold, let $(u^*,p^*)$ be saddle point. Then we have $(\bar{u}_{t}, \bar{p}_{t})\in(\mathbf{U}\cap\mathfrak{B}^{u})\times(\mathbf{C}^*\cap\mathfrak{B}^{p})$ and $\hat{u}(\bar{p}_{t})\in\mathbf{U}\cap\mathfrak{B}^{u}$, the following statements hold.
\begin{itemize}
\item[{\rm(i)}] Average point $(\bar{u}_t,\bar{p}_t)$ is an approximate saddle point of $L$:
$$-\frac{d_2}{\underline{\epsilon}(t+1)}+L(\bar{u}_{t},p)\leq L(\bar{u}_{t},\bar{p}_{t})\leq L(u,\bar{p}_{t})+\frac{d_2}{\underline{\epsilon}(t+1)}, \forall (u,p)\in(\mathbf{U}\cap\mathfrak{B}^{u})\times(\mathbf{C}^*\cap\mathfrak{B}^{p})$$
where $d_2=\max_{(u,p)\in(\mathbf{U}\cap\mathfrak{B}^{u})\times(\mathbf{C}^*\cap\mathfrak{B}^{p})}\big{[}D(u,u^0)+\frac{\epsilon^0}{2\gamma}\|p-p^{0}\|^2\big{]}$.
\item[{\rm(ii)}] Average point $(\bar{u}_t,\bar{p}_t)$ is an approximate saddle point of $L_\gamma$:
\begin{eqnarray*}
-\frac{r^pd_1+d_2}{\underline{\epsilon}(t+1)}-\frac{\gamma(d_1)^2}{2\underline{\epsilon}^2(t+1)^2}+L_{\gamma}(\bar{u}_{t},p)\leq L_{\gamma}(\bar{u}_{t},\bar{p}_{t})\leq L_{\gamma}(u,\bar{p}_{t})+\frac{r^pd_1+2d_2}{\underline{\epsilon}(t+1)}+\frac{\gamma(d_1)^2}{2\underline{\epsilon}^2(t+1)^2}.\\ \forall (u,p)\in(\mathbf{U}\cap\mathfrak{B}^{u})\times(\mathbf{C}^*\cap\mathfrak{B}^{p})
\end{eqnarray*}
\item[{\rm(iii)}] The existence on dual suboptimality is provided by average point $\bar{p}_t$:
\begin{eqnarray*}
    \psi_\gamma(p^*)\leq\psi_\gamma(\bar{p}_{t})+\frac{2r^pd_1+3d_2}{\underline{\epsilon}(t+1)}+\frac{\gamma (d_1)^2}{\underline{\epsilon}^2(t+1)^2}.
\end{eqnarray*}
\end{itemize}
\end{theorem}
\begin{proof}See Appendix {\bf A$_5$}.\qed
\end{proof}
\indent Therefore $(\bar{u}_{t},\bar{p}_{t})$ is an approximate saddle point of Lagrangian of (P) with accuracy of $O(1/t)$.
\subsection{Convergence rate analysis of VAPP for strongly convex problem (P)}
In this subsection, we consider strongly convex problem (P) where $G$ is strongly convex with modulus $\beta_G$. For the case where $J$ is strongly convex with modulus $\beta_J>0$ and $G$ is only convex, we can let $J\leftarrow J-\frac{\beta_J}{2}\|\cdot\|^2$ and $G\leftarrow G+\frac{\beta_J}{2}\|\cdot\|^2$. In order to obtain better convergence for solving (P), we modify the VAPP scheme with variable parameters as follows:
\begin{eqnarray}\label{a_parameter_choice}
\rho^k=(k+1)\eta\quad\mbox{and}\quad\epsilon^k=\frac{1}{(k+1)\eta\tau^2+B_G+B_{\Omega}+\beta_G},
\end{eqnarray}
with $\eta=\frac{\beta_G}{2\tau^2}$. Denote
\begin{equation}\label{eq:akbk}
a^k=(c_0+k)\left[\frac{1}{2\epsilon^k}-\frac{\beta_G}{2}\right],\quad\mbox{and}\quad b^k=\frac{c_0+k}{2\rho^k},
\end{equation}
with $c_0=\frac{2(B_G+B_{\Omega})}{\beta_G}+2$. Note that $c_0\geq1$, and by the definition of $\eta$, we have
\begin{equation}\label{eq:akbkbound}
a^k\geq\frac{\beta_G}{4}(k+1)^2,\quad\mbox{and}\quad b^k\geq\frac{1}{2\eta}.
\end{equation}
We modify VAPP for strongly convex case as VAPP-S as following. For simplicity, we take $K(u)=\frac{\|u\|^2}{2}$.
\begin{eqnarray}
\begin{array}{l}
\mbox{\bf VAPP-S Algorithm:}\\
\left\{\begin{array}{l}
u^{k+1}\leftarrow\min\limits_{u\in \mathbf{U}}\langle\nabla G(u^{k}), u\rangle + J(u)+ \langle\tilde{q}^k, \nabla\Omega(u^k)u+\Phi(u)\rangle+\frac{\|u-u^k\|^2}{2\epsilon^k};\\
p^{k+1}\leftarrow \Pi\big{(}p^{k}+\rho^k\Theta(u^{k+1})\big{)}.
\end{array}
\right.
\end{array}
\end{eqnarray}
where $\tilde{q}^k=\Pi\big{(}p^{k}+\rho^k\Theta(u^{k})\big{)}$. Let us consider a new iteration-based distance function $a^k\|u-u'\|^2+b^k\|p-p'\|^2$, the descent property of which is given by the following lemma.
\begin{lemma}\label{lemma:abound1} {\bf (Descent inequalities of generalized distance function for strongly convex (P))}
Let Assumptions~\ref{assump1} and~\ref{assump2} hold, $G$ is strongly convex with constant $\beta_G$, take parameters $\epsilon^k$ and $\rho^k$ satisfy~\eqref{a_parameter_choice}, and $\{(u^k,p^k)\}$ is generated by VAPP, for all $u\in\mathbf{U}$, $p\in\mathbf{C}^*$ and $k\in\mathbb{N}$, then it holds that
\begin{eqnarray*}
&&\bigg{\{}a^{k+1}\|u-u^{k+1}\|^2+b^{k+1}\|p-p^{k+1}\|^2\bigg{\}}-\bigg{\{}a^k\|u-u^k\|^2+b^k\|p-p^{k}\|^2\bigg{\}}\\
&\leq&(c_0+k)[L(u,\tilde{q}^k)-L(u^{k+1},p)]-\frac{c_0\beta_G}{2}\|u^k-u^{k+1}\|^2-\frac{1}{2\eta}\|\tilde{q}^k-p^k\|^2
\end{eqnarray*}
\end{lemma}
\proof From the strongly convexity, the assertion is derived easily by the similar arguments in proof of Lemma~\ref{lemma:bound1} (see {\bf A$_1$} in Appendix).
\qed
Based Lemma~\ref{lemma:abound1}, we establish the following convergence analysis of VAPP-S for strongly convex problem.
\begin{theorem}[Convergence analysis of VAPP-S for strongly convex (P)]\label{theo:aconvergence}
Let assumptions of Lemma~\ref{lemma:abound1} hold, then the sequence $\{(u^{k},p^{k})\}$ generated by VAPP-S is bounded and converges to $(u^{*},p^{*})$, which is the saddle point of $L$ over $\mathbf{U}\times\mathbf{C}^*$
\end{theorem}
\proof
Taking $u=u^{*}$ and $p=p^*$ in Lemma~\ref{lemma:abound1}, we conclude that the sequence $a^k\|u^*-u^k\|^2+b^k\|p^*-p^{k}\|^2$ is strictly decreasing, unless $u^k=u^{k+1}$ and $p^k=\tilde{q}^k$ or $p^k=p^{k+1}$. The desired result is derived by a similar argument of~\cite{CohenZ}.
\qed
For any integer number $t$, let $\bar{u}_{t}=\frac{\sum_{k=0}^{t}(c_0+k)u^{k+1}}{\sum_{k=0}^{t}(c_0+k)}$ and $\bar{p}_{t}=\frac{\sum_{k=0}^{t}(c_0+k)\tilde{q}^{k}}{\sum_{k=0}^{t}(c_0+k)}$. Obviously that $\sum_{k=0}^{t}(c_0+k)=\frac{1}{2}(t+1)(t+2c_0)$. Therefore, we have that $(\bar{u}_t,\bar{p}_t)\in\mathbf{U}\times\mathbf{C}^*$ and $(u^*,p^*)\in\mathbf{U}\times\mathbf{C}^*$.
Then we can get the following convergence rate analysis.
\begin{theorem}[Primal error bound, bifunction value, primal suboptimality and feasibility of VAPP-S for strongly convex (P)]\label{theo:aconvergence_rate}
Let assumptions of Lemma~\ref{lemma:abound1} hold, then
\begin{itemize}
\item[{\rm(i)}] Global estimate in primal error bound value:
$$\|u^*-u^{t}\|^2\leq o(1/t^2);$$
\item[{\rm(ii)}] Global estimate in bifunction value of (EP):
\begin{equation}\label{GEB}
L(\bar{u}_t,p)-L(u,\bar{p}_t)\leq\frac{2a^0\|u-u^0\|^2+2b^0\|p-p^{0}\|^2}{(t+1)(t+2c_0)},\quad\forall (u,p)\in\mathbf{U}\times\mathbf{C}^*.
\end{equation}
\item[{\rm(iii)}] Feasibility: $$\|\Pi\big{(}\Theta(\bar{u}_{t})\big{)}\|\leq O(1/t^2).$$
\item[{\rm(iv)}] Primal suboptimality: $$-O(1/t^2)\leq(G+J)(\bar{u}_{t})-(G+J)(u^*)\leq O(1/t^2).$$
\end{itemize}
\end{theorem}
\proof
\begin{itemize}
\item[{\rm(i)}] From the convergence Theorem~\ref{theo:aconvergence}, we have that
\begin{eqnarray}\label{eq:convergence_2}
\lim_{t\rightarrow\infty}a^{t}\|u^*-u^{t}\|^2+b^{t}\|p^*-p^{t}\|^2=0.
\end{eqnarray}
Since $a^k$ satisfy~\eqref{eq:akbkbound}, we have that $a^{t}\geq\frac{\beta_G}{4}(t+1)^2$, it follows that
$$\|u^{t}-u^*\|^2=o(1/t^2).$$
\item[{\rm(ii-iv)}] Using Lemma~\ref{lemma:abound1} and the same arguments in the proof of Theorem~\ref{thm:ergodic_iteration_complexity}, we can show that the statements (ii)-(iv) hold.
\end{itemize}
\qed
\section{Linear convergence of VAPP with various error bounds conditions}\label{sec:linear_convergence}
\indent In this section, we study the error bound conditions to ensure the linear convergence of VAPP. \\
\indent The saddle point $(u,p)$ of Lagrangian of problem (P) satisfies the following KKT system:
\begin{equation}\label{eq:KKT}
\left\{
\begin{array}{l}
0\in\nabla G(u)+\partial J(u) +\left(\nabla\Omega(u)+\partial\Phi(u)\right)^{\top}p+\mathcal{N}_{\mathbf{U}}(u)      \\
0\in-\Theta(u)+\mathcal{N}_{\mathbf{C}^*}(p),
\end{array}
\right.
\end{equation}
where $\mathcal{N}_{\mathbf{U}}(u):=\{\xi:\langle\xi,\zeta-u\rangle\leq0, \forall\zeta\in\mathbf{U}\}$ is the normal cone at $u$ to a given convex set $\mathbf{U}$. It is natural to define the Lagrangian based KKT mapping $H:\RR^n\times\RR^m\rightrightarrows\RR^n\times\RR^m$ as:\\
\begin{equation}\label{eq:Hw}
H(w)=\left(
\begin{array}{l}
\nabla G(u)+\partial J(u) +\left(\nabla\Omega(u)+\partial\Phi(u)\right)^{\top}p+\mathcal{N}_{\mathbf{U}}(u)      \\
-\Theta(u)+\mathcal{N}_{\mathbf{C}^*}(p)
\end{array}
\right)
\end{equation}
with $w=\left(\begin{array}{l}u\\p\end{array}\right)$. Thus, KKT system~\eqref{eq:KKT} can be presented as a inclusion problem $0\in H(w)$. For $H(w)$ given in~\eqref{eq:Hw}, its inverse mapping is $H^{-1}(v)=\{w|v\in H(w)\}$. Under Assumption~\ref{assump1}, the set of saddle points $\mathbf{S}^*\neq\emptyset$ and is equal to $H^{-1}(0)$.\\
\indent The primal-dual pair $(u^*,p^*)\in\mathbf{S}^*$ also satisfies the augmented Lagrangian based KKT system:
\begin{equation}\label{eq:KKTgamma}
\left\{
\begin{array}{l}
0\in\nabla G(u)+\partial J(u) +(\nabla\Omega(u)+\partial\Phi(u))^{\top}\Pi(p+\gamma\Theta(u))+\mathcal{N}_{\mathbf{U}}(u)      \\
0\in-\nabla\psi_{\gamma}(p)+\mathcal{N}_{\mathbf{C}^*}(p)=-\Theta(u)+\mathcal{N}_{\mathbf{C}^*}(p)
\end{array}
\right.
\end{equation}
The following mapping is referred to as augmented Lagrangian-based KKT mapping:
\begin{eqnarray*}\label{eq:Hgammaw}
H_{\gamma}(w)=\left(
\begin{array}{l}
\nabla G(u)+\partial J(u) +\left(\nabla\Omega(u)+\partial\Phi(u)\right)^{\top}\Pi(p+\gamma\Theta(u))+\mathcal{N}_{\mathbf{U}}(u)      \\
-\Theta(u)+\mathcal{N}_{\mathbf{C}^*}(p)
\end{array}
\right)
\end{eqnarray*}

We define the generated distance function for a point to set with respect to Bregman function $D(v,u)$ as follows:
$$dist_{D,\epsilon^k}(w,\mathbf{S}^*)=\min_{w^*\in\mathbf{S}^*}[D(u^*,u)+\frac{\epsilon^k}{2\gamma}\|p-p^*\|^2]^{\frac{1}{2}},$$
The classic distance function for a point to set is
$$dist(w,\mathbf{S}^*)=\min_{w^*\in\mathbf{S}^*}[\|u-u^*\|^2+\|p-p^*\|^2]^{\frac{1}{2}}.$$
By Assumption~\ref{assump2} for $D$ and~\eqref{eq:parameter} of $\epsilon^k$, there are $\mathfrak{b}_1$ and $\mathfrak{b_2}$ such that
\begin{eqnarray}\label{dist-bound}
\mathfrak{b}_1dist(w,\mathbf{S}^*)\leq dist_{D,\epsilon^k}(w,\mathbf{S}^*)\leq\mathfrak{b}_2dist(w,\mathbf{S}^*).
\end{eqnarray}
Denote that $\mathbb{B}(x^*;\eta):=\{x:\|x-x^*\|\leq\eta\}$. Now we present the VAPP-iteration-based error bound (V-IEB) which guarantees the linear convergence of VAPP.
\begin{definition}[VAPP-iteration-based error bound (V-IEB)] Let $\{w^k\}$ be the primal-dual sequence generated by the VAPP converges to $w^*\in\mathbf{S}^*$. If there exists $c_1>0$ and $\eta>0$ such that
\begin{equation}\label{eq:VEB}
dist(w^{k+1},\mathbf{S}^*)\leq c_1\|w^k-w^{k+1}\|,\quad\mbox{when}\quad w^{k+1}\in\mathbb{B}(w^*;\eta)
\end{equation}
then $\{w^k\}$ is said to satisfy a VAPP-iteration-based error bound condition.
\end{definition}
With V-IEB, we can prove the linear convergence of VAPP by the following theorem.
\begin{theorem}[V-IEB implies global linear convergence] Suppose Assumption~\ref{assump1} and~\ref{assump2} hold. Let $\{w^k\}$ be the sequence generated by the VAPP converges to $w^*$ which satisfies the V-IEB condition~\eqref{eq:VEB}, then there exists $\beta\in(0,1)$ and $\eta>0$ such that
\begin{equation}
dist_{D,\epsilon^{k+1}}^2(w^{k+1},\mathbf{S}^*)\leq\beta\cdot dist_{D,\epsilon^{k}}^2(w^k,\mathbf{S}^*),\quad\forall k.
\end{equation}
\end{theorem}
\begin{proof}
Let $\{w^k\}$ be the sequence generated by VAPP. For given $w^k=(u^k,p^k)$, let $w_k^*=(u_k^*,p_k^*)=\arg\min\limits_{w^*\in\mathbf{S}^*}[D(u^*,u^k)+\frac{\epsilon^k}{2\gamma}\|p^k-p^*\|^2]^{\frac{1}{2}}$ by Lemma~\ref{lemma:bound1} with $u=u_k^*$ and $p=p_k^*$, then it follows that
\begin{eqnarray}\label{eq:linear2}
&&\big{[}D(u_k^*,u^k)+\frac{\epsilon^k}{2\gamma}\|p_k^*-p^{k}\|^2\big{]}-\big{[}D(u_k^*,u^{k+1})+\frac{\epsilon^{k+1}}{2\gamma}\|p_k^*-p^{k+1}\|^2\big{]}\nonumber\\
&\geq&\frac{\beta-\overline{\epsilon}(B_G+B_{\Omega}+\gamma\tau^2)}{2}\|u^k-u^{k+1}\|^2+\frac{\underline{\epsilon}}{2\gamma}\|p^k-q^k\|^2\nonumber\\
&\geq&\alpha (c_1)^2[(1+2\gamma^2\tau^2)\|u^k-u^{k+1}\|^2+2\|p^{k}-q^k\|^2]\nonumber\\
&\geq&\alpha (c_1)^2[\|u^k-u^{k+1}\|^2+2(\|p^{k+1}-q^k\|^2+\|p^{k}-q^k\|^2)]\nonumber\\
&&\qquad\qquad\qquad\qquad\qquad\qquad\;\mbox{(since $\|p^{k+1}-q^k\|\leq\gamma\tau\|u^{k}-u^{k+1}\|$)}\nonumber\\
&\geq&\alpha (c_1)^2[\|u^k-u^{k+1}\|^2+\|p^k-p^{k+1}\|^2]\nonumber\\
&=&\alpha (c_1)^2\|w^k-w^{k+1}\|^2
\end{eqnarray}
where $\alpha=\min\{\frac{\beta-\overline{\epsilon}(B_G+B_{\Omega}+\gamma\tau^2)}{2},\frac{\underline{\epsilon}}{2\gamma}\}/\big{(}(c_1)^2\max\{1+2\gamma^2\tau^2,2\}\big{)}>0$. By the V-IEB condition, there exists $c_1>0$ and $\eta>0$ such that
\begin{equation}\label{VAPP-EB}
dist(w^{k+1},\mathbf{S}^*)\leq c_1\|w^k-w^{k+1}\|,\quad\mbox{when}\quad w^{k+1}\in\mathbb{B}(w^*;\eta)
\end{equation}
Together~\eqref{dist-bound},~\eqref{eq:linear2} and~\eqref{VAPP-EB}, subsequently, we have that
\begin{eqnarray*}
&&\alpha dist_{D,\epsilon^{k+1}}^2(w^{k+1},\mathbf{S}^*)\nonumber\\
&\leq&\alpha(\mathfrak{b}_2)^2dist^2(w^{k+1},\mathbf{S}^*)\quad\mbox{(by~\eqref{dist-bound}.)}\nonumber\\
&\leq&\alpha(\mathfrak{b}_2)^2(c_1)^2\|w^k-w^{k+1}\|^2\quad\mbox{(by~\eqref{VAPP-EB})}\nonumber\\
&\leq&(\mathfrak{b}_2)^2[dist_{D,\epsilon^k}^2(w^k,\mathbf{S}^*)-dist_{D,\epsilon^{k+1}}^2(w^{k+1},\mathbf{S}^*)].\quad\mbox{(by~\eqref{eq:linear2})}
\end{eqnarray*}
It follows the local linear convergence of VAPP
\begin{equation}
dist_{D,\epsilon^{k+1}}^2(w^{k+1},\mathbf{S}^*)\leq\beta'\cdot dist_{D,\epsilon^k}^2(w^k,\mathbf{S}^*),\quad\mbox{when $w^{k+1}\in\mathbb{B}(w^*;\eta)$}
\end{equation}
with $\beta'=(\mathfrak{b}_2)^2/(\alpha+(\mathfrak{b}_2)^2)\in(0,1)$.

By the fact that $\{w^k\}$ converges to $w^*$, it easily follows that for any $\eta>0$, there is $\tilde{\eta}>0$ such that
$$\|w^k-w^{k+1}\|\leq\tilde{\eta}\Rightarrow w^{k+1}\in\mathbb{B}(w^*;\eta).$$
Using the same argument of Proposition 6.1.2 in~\cite{JSPang2007}, we obtain the global linear convergence of VAPP. That is, there is $\beta\in(0,1)$ such that
$$dist_{D,\epsilon^{k+1}}^2(w^{k+1},\mathbf{S}^*)\leq\beta\cdot dist_{D,\epsilon^k}^2(w^k,\mathbf{S}^*)\quad\forall k.\qquad\qquad\qquad\qed$$
\end{proof}
We introduce the following stability notions of set valued mapping which will play a key role to guarantee V-IEB holding.
\begin{definition}\qquad\\
\begin{itemize}
\item[(i)] ({\bf Metric subregularity}) The set-valued mapping $\mathcal{F}(w)$ is called metric subregular around $(w^*,0)$ if $\exists\mathbb{B}(w^*;\eta)$ of $w^*$ and $c_2>0$ such that
\begin{equation}\label{MS}
dist(w,\mathcal{F}^{-1}(0))\leq c_2dist\left(0,\mathcal{F}(w)\right),\quad\forall w\in\mathbb{B}(w^*;\eta)
\end{equation}
\item[(ii)] ({\bf Calmness of $\mathcal{F}^{-1}$, Ye and Ye~\cite{Ye1997}, Rockafellar and Wets~\cite{Rockafellar2009}}) The set-valued mapping $\mathcal{F}^{-1}$ is calmness $(0,w^*)$ if there exists a neighborhood $\mathbb{B}(w^*;\delta)$ of $w^*$ and $\kappa>0$ such that
$$\mathcal{F}^{-1}(v)\cap\mathbb{B}(w^*,\delta)\subset\mathcal{F}^{-1}(0)+\kappa\|v\|\cdot\mathbb{B}(0;1), \forall v\in\mathbb{B}(0;\delta).$$
\item[(iii)] ({\bf Local upper-Lipschitz for $\mathcal{F}^{-1}$, Robinson, 1981~\cite{Robinson1981}}) The set-valued map $\mathcal{F}^{-1}$ is local upper-Lipschitz for $\mathcal{F}^{-1}$ at $0$ if there exists a neighborhood $\mathbb{B}(0;\delta)$ of $0$ and $\kappa>0$ such that
$$\mathcal{F}^{-1}(v)\subset \mathcal{F}^{-1}(0)+\kappa\|v\|\cdot\mathbb{B}(0;1), \forall v\in\mathbb{B}(0;\delta).$$
\item[(iv)] ({\bf Pseudo-Lipschitz (Aubin property) for $\mathcal{F}^{-1}$, Aubin, 1984~\cite{Aubin1984}}) The mapping $\mathcal{F}^{-1}$ is pseudo-Lipschitz continuous around $(0,w^*)$ if there exists neighborhood $\mathbb{B}(0;\delta)$ of $0$ and $\mathbb{B}(w^*;\delta)$ of $w^*$ and $\kappa>0$ such that
    $$\mathcal{F}^{-1}(v)\cap\mathbb{B}(w^*;\delta)\subset\mathcal{F}^{-1}(v')+\kappa\|v-v'\|\cdot\mathbb{B}(0;1), \forall v,v'\in\mathbb{B}(0;\delta).$$
\item[(v)] ({\bf Lipschitz for $\mathcal{F}^{-1}$, Rockafellar, 1976~\cite{Rock76Monotone}}) The mapping $\mathcal{F}^{-1}$ is Lipschitz continuous at $0$ if there exist neighborhood $\mathbb{B}(0;\delta)$ of $0$ and $\kappa>0$ such that
$$\|\mathcal{F}^{-1}(v)-\mathcal{F}^{-1}(0)\|\leq\kappa\|v\|,\quad\forall v\in\mathbb{B}(0;\delta).$$
\end{itemize}
\end{definition}

The relationship among the V-IEB, metric subregularity and other stability of set-valued mapping is shown in Figure~\ref{fig:1}. (also see Ye and Zhou~\cite{Ye18}, Dontchev and Rockafellar~\cite{DontchevRockafellar2009})
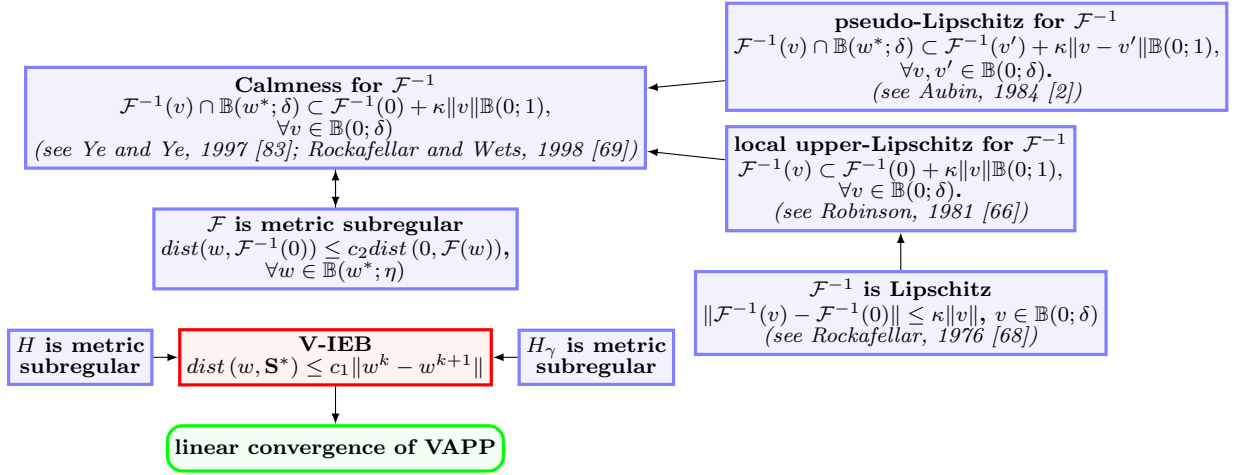
\begin{figure}
\centering
\begin{center}
\scriptsize
		\tikzstyle{format}=[rectangle,draw,thin,fill=white]
		\tikzstyle{test}=[diamond,aspect=2,draw,thin]
		\tikzstyle{point}=[coordinate,on grid,]
\begin{tikzpicture}
[%?????????latex ?????
>=latex,
%???????????
node distance=5mm,
% hv path ???????????????????????????????vh ????skip loop ???
%???-??-??? vskip loop ?????-???-??
 ract/.style={draw=blue!50, fill=blue!5,rectangle,minimum size=6mm, very thick, font=\itshape, align=center},
 racc/.style={rectangle, align=center},
 ractm/.style={draw=red!100, fill=red!5,rectangle,minimum size=6mm, very thick, font=\itshape, align=center},
 cirl/.style={draw, fill=yellow!20,circle,   minimum size=6mm, very thick, font=\itshape, align=center},
 raco/.style={draw=green!500, fill=green!5,rectangle,rounded corners=2mm,  minimum size=6mm, very thick, font=\itshape, align=center},
 hv path/.style={to path={-| (\tikztotarget)}},
 vh path/.style={to path={|- (\tikztotarget)}},
 skip loop/.style={to path={-- ++(0,#1) -| (\tikztotarget)}},
 vskip loop/.style={to path={-- ++(#1,0) |- (\tikztotarget)}}]
        \node (a) [ractm]{\baselineskip=3pt\small {\bf V-IEB}\\ \baselineskip=3pt\footnotesize$dist\left(w,\mathbf{S}^*\right)\leq c_1\|w^k-w^{k+1}\|$};
        \node (aa) [raco, below = of a]{\baselineskip=3pt\small {\bf linear convergence of VAPP}};
        \node (aa1) [ract, left = of a, xshift=5]{\baselineskip=3pt\small {\bf $H$ is metric}\\
                                        \baselineskip=3pt\small {\bf subregular}};
        \node (aa2) [ract, right = of a, xshift=-5]{\baselineskip=3pt\small {\bf $H_{\gamma}$ is metric}\\
                                        \baselineskip=3pt\small {\bf subregular}};
        \node (b) [ract, above = of a]{\baselineskip=3pt\small {\bf $\mathcal{F}$ is metric subregular}\\
                                       \baselineskip=3pt\small {\bf $dist(w,\mathcal{F}^{-1}(0))\leq c_2dist\left(0,\mathcal{F}(w)\right)$,}\\
                                       \baselineskip=3pt\small {\bf $\forall w\in\mathbb{B}(w^*;\eta)$}};
        \node (bb) [ract, above = of b]{\baselineskip=3pt\small {\bf Calmness for $\mathcal{F}^{-1}$}\\
                                        \baselineskip=3pt\small {\bf $\mathcal{F}^{-1}(v)\cap\mathbb{B}(w^*;\delta)\subset \mathcal{F}^{-1}(0)+\kappa\|v\|\mathbb{B}(0;1),$}\\
                                        \baselineskip=3pt\small {\bf $\forall v\in\mathbb{B}(0;\delta)$}\\
                                        \baselineskip=3pt\small {(see Ye and Ye, 1997~\cite{Ye1997}; Rockafellar and Wets, 1998~\cite{Rockafellar2009})}};
        \node (c) [ract, right = of bb, xshift=15, yshift=23]{\baselineskip=3pt\small {\bf pseudo-Lipschitz for $\mathcal{F}^{-1}$}\\
                                       \baselineskip=3pt\small {\bf $\mathcal{F}^{-1}(v)\cap\mathbb{B}(w^*;\delta)\subset \mathcal{F}^{-1}(v')+\kappa\|v-v'\|\mathbb{B}(0;1),$}\\
                                       \baselineskip=3pt\small {\bf $\forall v,v'\in\mathbb{B}(0;\delta)$.}\\
                                       \baselineskip=3pt\footnotesize {(see Aubin, 1984~\cite{Aubin1984})}};
        \node (d) [ract, right = of bb, xshift=15, yshift=-23]{\baselineskip=3pt\small {\bf local upper-Lipschitz for $\mathcal{F}^{-1}$}\\
                                       \baselineskip=3pt\small {\bf $\mathcal{F}^{-1}(v)\subset \mathcal{F}^{-1}(0)+\kappa\|v\|\mathbb{B}(0;1),$}\\
                                       \baselineskip=3pt\small {\bf $\forall v\in\mathbb{B}(0;\delta)$.}\\
                                       \baselineskip=3pt\footnotesize {(see Robinson, 1981~\cite{Robinson1981})}};
        \node (e) [ract, below = of d]{\baselineskip=3pt\small {\bf $\mathcal{F}^{-1}$ is Lipschitz}\\
                                       \baselineskip=3pt\small {\bf $\|\mathcal{F}^{-1}(v)-\mathcal{F}^{-1}(0)\|\leq\kappa\|v\|$, $v\in\mathbb{B}(0;\delta)$}\\
                                       \baselineskip=3pt\footnotesize {(see Rockafellar, 1976~\cite{Rock76Monotone})}};
        \path %(a) edge[->] (b)
               (a) edge[->] (aa)
              %(b) edge[->] (a)
             (aa1) edge[->] (a)
             (aa2) edge[->] (a)
               (c) edge[->] (bb)
               (d) edge[->] (bb)
               (bb) edge[->] (b)
               (b) edge[->] (bb)
               (e) edge[->] (d);
\end{tikzpicture}
\caption{The relationship among the notions of the metric subregularity and other stability of set-valued mapping.}\label{fig:1}
\end{center}
\end{figure}

The following proposition gives a sufficient condition for V-IEB.

\begin{proposition}[Metric subregularity of $H(w)$ or $H_{\gamma}(w)$ implies V-IEB]\label{prop:MSH-VIEB}
Suppose Assumptions~\ref{assump1} and~\ref{assump2} hold. Let $\{w^k\}$ be the sequence generated by the VAPP converges to $w^*$. If one of the following condition holds, then the sequence $\{w^k\}$ satisfies a V-IEB condition.
\begin{itemize}
\item[(i)] $H(w)$ is metric subregular around $(w^*,0)$;
\item[(ii)] $H_{\gamma}(w)$ is metric subregular around $(w^*,0)$;
\end{itemize}
\end{proposition}
\begin{proof}
\begin{itemize}
\item[(i)] By VAPP scheme, we have
\begin{equation}
\left\{
\begin{array}{l}
0\in\nabla G(u^k)+\partial J(u^{k+1}) +\big{(}\nabla\Omega(u^{k})+\partial\Phi(u^{k+1})\big{)}^{\top} q^k+\frac{1}{\epsilon^k}\left[\nabla K(u^{k+1})-\nabla K(u^{k})\right]+\mathcal{N}_{\mathbf{U}}(u^{k+1})      \\
0\in-\Theta(u^{k+1})+\frac{1}{\gamma}\left[p^{k+1}-p^{k}\right]+\mathcal{N}_{\mathbf{C}^*}(p^{k+1})
\end{array}
\right.
\end{equation}
Thus
\begin{eqnarray*}
v^{k+1}=\left(
\begin{array}{l}
\nabla G(u^{k+1})-\nabla G(u^k) +\big{(}\theta^{k+1}\big{)}^{\top}(p^{k+1}-q^k)\\
\qquad+\big{(}\nabla\Omega(u^{k+1})-\nabla\Omega(u^{k})\big{)}^{\top}q^k+\frac{1}{\epsilon^k}\left[\nabla K(u^{k})-K(u^{k+1})\right]      \\
\frac{1}{\gamma}\left[p^k-p^{k+1}\right]
\end{array}
\right)\in H(w^{k+1})
\end{eqnarray*}
with $\theta^{k+1}\in\partial\Theta(u^{k+1})$. From Assumption~\ref{assump1}, there are positive numbers $\mathfrak{a}$ and $\mathfrak{b}$ such that
\begin{eqnarray}\label{MS-VEB}
\|v^{k+1}\|^2&\leq&\mathfrak{a}\|u^k-u^{k+1}\|^2+\mathfrak{b}\|p^k-p^{k+1}\|^2\nonumber \\
                               &\leq&\max\{\mathfrak{a},\mathfrak{b}\}\|w^k-w^{k+1}\|^2.
\end{eqnarray}
Since $H(w)$ is metric subregular around $(w^*,0)$, then
\begin{eqnarray}\label{MS+}
dist(w^{k+1},\mathbf{S}^*)&\leq& c_2dist(0,H(w^{k+1}))\nonumber\\
&\leq&c_2dist(0,v^{k+1})\nonumber\\
&\leq&c_2\sqrt{\max\{\mathfrak{a},\mathfrak{b}\}}\|w^k-w^{k+1}\|,\quad\forall w^{k+1}\in\mathbb{B}(w^*;\eta).
\end{eqnarray}
which shows $\{w^k\}$ satisfies V-IEB condition.
\item[(ii)] The proof is similar to (i).\qed
\end{itemize}
\end{proof}
Next, we give certain instances with the metric subregularity holding.
\begin{proposition} Consider problem (P), and suppose Assumptions~\ref{assump1} and~\ref{assump2} hold. Let $w^*=(u^*,p^*)$ be the saddle point of (P). The following assertions hold:
\begin{itemize}
\item[{\rm(i)}] $G(u)$ is strongly convex on $\mathbf{U}$, $\mathbf{C}=\{0\}$ or problem (P) only has equality constraints $\Theta(u)=Au-b=0$. Then $H_{\gamma}(w)$ is metric subregular around $(w^*,0)$.
\item[{\rm(ii)}] $\nabla G(u)$ and $\partial J(u)$ are piecewise linear functions, $\mathbf{U}$ is polyhedral, $\Theta(u)=Au-b$, and $\mathbf{C}=\{0\}$. Then $H(w)$ is metric subregular around $(w^*,0)$.
\item[{\rm(iii)}] $G(u)=\frac{1}{2}\langle u,Qu\rangle+\langle c,u\rangle$, $Q\in\RR^{n\times n}$ is symmetric p.s.d matrix, $c\in\RR^n$, $\mathbf{U}$ is polyhedral, $\Theta(u)=Au-b$, and $\mathbf{C}$ is polyhedral convex cone in $\RR^n$. Then $H(w)$ is metric subregular around $(w^*,0)$.
\end{itemize}
\end{proposition}
\begin{proof}
\begin{itemize}
\item[(i)] In this case, the augmented Lagrangian function is
$$L_{\gamma}(u,p)=G(u)+J(u)+\langle p,Au-b\rangle+\frac{\gamma}{2}\|Au-b\|^2.$$
The saddle point problem of (P) can be reformulated as the following inclusion problem:
$$0\in H_{\gamma}(w)=\left(\begin{array}{c}\nabla G(u)+\partial J(u)+\gamma A^{\top}(Au-b)+A^{\top}p+\mathcal{N}_{\mathbf{U}}(u)\\-\nabla\psi_{\gamma}(p)\end{array}\right)$$
By a similar argument of claim 6.1 in~\cite{HongLuo2017}, there is $\delta>0$ and $\tau>0$, such that
\begin{equation}\label{eq:dual_EB}
\|\hat{u}(p)-u^*\|^2+\|p-p^*\|^2\leq\tau\|\nabla\psi_{\gamma}(p)\|^2\quad\mbox{when}\quad\|\nabla\psi_{\gamma}(p)\|\leq\delta.
\end{equation}
where $\hat{u}(p)=\arg\min\limits_{u\in\mathbf{U}}L_{\gamma}(u,p)$, and $(u^*,p^*)$ is a saddle point of (P). From~\cite{HongLuo2017}, $\nabla\psi_{\gamma}(p)$ is Lipschitz; thus there is $\eta$ such that~\eqref{eq:dual_EB} holds for $p\in\mathbb{B}(p^*;\eta)$. The strong convexity of $G$ with fact $\hat{u}(p)=\arg\min\limits_{u\in\mathbf{U}}L_{\gamma}(u,p)$ follows that
$$\langle\nabla G(u)+\xi+A^{\top}(Au-b)+A^{\top}p+\nu,u-\hat{u}(p)\rangle\geq\beta_G\|u-\hat{u}(p)\|^2,\quad\forall\xi\in\partial J(u), \forall\nu\in\mathcal{N}_{\mathbf{U}}(u)$$
Thus
\begin{equation}\label{eq:primal-EB}
\|\nabla G(u)+\xi+A^{\top}(Au-b)+A^{\top}p+\nu\|^2\geq\beta_G^2\|u-\hat{u}(p)\|^2,\quad\forall\xi\in\partial J(u), \forall\nu\in\mathcal{N}_{\mathbf{U}}(u).
\end{equation}
Combining~\eqref{eq:dual_EB} and~\eqref{eq:primal-EB}, $\forall p\in\mathbb{B}(p^*;\eta)$, there is $\theta>0$ such that
\begin{eqnarray*}
dist(0,H_{\gamma}(w))&\geq&\theta\sqrt{\left(\|u-u^*\|^2+\|p-p^*\|^2\right)}\\
&\geq&\theta dist(w,H_{\gamma}^{-1}(0)),\quad\mbox{for $w\in\mathbb{B}(w^*;\eta)$.}
\end{eqnarray*}
Therefore $H_{\gamma}(w)$ is metric subregular around $(w^*,0)$.
\item[(ii)] The claim is provided by the error bound result established in Theorem 3.3 of~\cite{Zheng2014}.
\item[(iii)] See Proposition 1 of~\cite{Robinson1981}.
\end{itemize}
\qed
\end{proof}
For the problem with nonlinear constraints, some verifiable sufficient conditions for the error bounds of KKT system mapping are given in~\cite{Ye18} and~\cite{Dontchev}. However, in general, these conditions are not easy to check.
\section{A view of Forward-Backward Splitting for VAPP and the connection with various primal-dual splitting algorithms}\label{sec:FBS-relation}
\subsection{A view of Forward-Backward Splitting (FBS) for VAPP}
In this subsection, we will show that VAPP algorithm can be derived from FBS for inclusion problem of (P). For simplicity, we consider problem (P) with differentiable term $\Phi$ in constraints. Recall the augmented Lagrangian function of (P) is
$$L_{\gamma}(u,p)=G(u)+J(u)+\varphi\left(\Theta(u),p\right).$$
By the definition, the saddle point $(u,p)\in\mathbf{U}\times\mathbf{C}^*$ of $L_{\gamma}$ satisfies
\begin{equation}\label{eq:uprimal}
0\in\partial_uL_{\gamma}(u,p)+\mathcal{N}_{\mathbf{U}}(u)
\end{equation}
and
\begin{equation}\label{eq:pdual}
0\in-\nabla_pL_{\gamma}(u,p).
\end{equation}
Thus, the saddle point problem of (P) can be represented as the following inclusion problem:
\begin{equation}
0\in H_{\gamma}(w)=\left(
\begin{array}{c}
\partial_uL_{\gamma}(u,p)+\mathcal{N}_{\mathbf{U}}(u) \\
-\nabla_pL_{\gamma}(u,p)
\end{array}
\right).
\end{equation}
To find the connection between VAPP algorithm and FBS, we decompose $H_{\gamma}(w)$ as $H_{\gamma}=A+B$, where
\begin{equation}
A(w)=\left(
\begin{array}{c}
\partial J(u)+\mathcal{N}_{\mathbf{U}}(u)\\
\mathbf{0}_m
\end{array}
\right)
\end{equation}
and
\begin{equation}
B(w)=\left(
\begin{array}{c}
\nabla G(u)+\nabla_u\varphi\left(\Theta(u),p\right)\\
-\nabla_p\varphi\left(\Theta(u),p\right)
\end{array}
\right)=\left(\begin{array}{c}\nabla G(u)+\left(\nabla\Omega(u)+\nabla\Phi(u)\right)^{\top}\Pi\left(p+\gamma\Theta(u)\right)\\
-\frac{1}{\gamma}\bigg{[}\Pi\big{(}p+\gamma\Theta(u)\big{)}-p\bigg{]}\end{array}\right).
\end{equation}
For finding the saddle point of (P), we only need to solve the inclusion problem:
\begin{equation}\label{eq:inclusion}
0\in A(w)+B(w)
\end{equation}
Obviously, both $A(w)$ and $B(w)$ are maximal monotone (see Lemma 3.2 in~\cite{Zhu2003}). Given $w^k$, we introduce nonlinear Bregman operator as $\Gamma^k(w)=\left(\begin{array}{c}\frac{1}{\epsilon^k}\nabla K(u)+\left(\nabla\Phi(u)\right)^{\top}q^k\\ \frac{1}{\gamma}\bigg{[}p-\Pi\big{(}p^k+\gamma\Theta(u)\big{)}\bigg{]}\end{array}\right)$ with $q^k=\Pi\left(p^k+\gamma\Theta(u^k)\right)$. Here we briefly prove the strong monotoncity of $\Gamma^k$ on $\mathbf{U}\times\RR^m$. For any $w,w'\in\mathbf{U}\times\RR^m$, we have that
\begin{eqnarray*}
\langle \Gamma^k(w)-\Gamma^k(w'),w-w'\rangle&=&\langle\frac{1}{\epsilon^k}\nabla K(u)+\left(\nabla\Phi(u)\right)^{\top}q^k-\frac{1}{\epsilon^k}\nabla K(u')-\left(\nabla\Phi(u')\right)^{\top}q^k,u-u' \rangle\nonumber\\
&&+\langle\frac{1}{\gamma}\bigg{[}p-\Pi\big{(}p^k+\gamma\Theta(u)\big{)}\bigg{]}-\frac{1}{\gamma}\bigg{[}p'-\Pi\big{(}p^k+\gamma\Theta(u')\big{)}\bigg{]},p-p'\rangle\nonumber\\
&\geq&\frac{\beta}{\overline{\epsilon}}\|u-u'\|^2+\frac{1}{\gamma}\|p-p'\|^2-\tau\|u-u'\|\cdot\|p-p'\|\nonumber\\
&\geq&\frac{\gamma\tau^2}{2}\|u-u'\|^2+\frac{1}{2\gamma}\|p-p'\|^2-\tau\|u-u'\|\cdot\|p-p'\|\nonumber\\
&&+\frac{\beta}{2\overline{\epsilon}}\|u-u'\|^2+\frac{1}{2\gamma}\|p-p'\|^2\qquad\mbox{(by $\overline{\epsilon}\leq\frac{\beta}{\gamma\tau^2}$ in~\eqref{eq:parameter})}\nonumber\\
&\geq&\frac{\beta}{2\overline{\epsilon}}\|u-u'\|^2+\frac{1}{2\gamma}\|p-p'\|^2.
\end{eqnarray*}
Now we propose the iteration based nonlinear forward-backward splitting algorithm to solve~\eqref{eq:inclusion}:
\begin{equation}\label{eq:FBS}
w^{k+1}=(\Gamma^k+ A)^{-1}(\Gamma^k- B)w^{k},
\end{equation}
which consists of first applying a forward (explicit) step and then a backward (implicit) step. By~\eqref{eq:FBS}, it follows that
\begin{equation*}
(\Gamma^k-B)w^{k}\in(\Gamma^k+A)w^{k+1}.
\end{equation*}
Finally, we obtain
\begin{equation*}\label{FBS-VAPP-0}
0\in\left(\begin{array}{l}\frac{1}{\epsilon^k}[\nabla K(u^{k+1})-\nabla K(u^k)]+\nabla G(u^k)+\left(\nabla\Omega(u^k)\right)^{\top}q^k+\partial J(u^{k+1})+\left(\nabla\Phi(u^{k+1})\right)^{\top}q^k+\mathcal{N}_{\mathbf{U}}(u^{k+1})\\
p^{k+1}-\Pi\left(p^k+\gamma\Theta(u^{k+1})\right)\end{array}\right).
\end{equation*}
Therefore,
\begin{eqnarray}
u^{k+1}&=&\arg\min\limits_{u\in\mathbf{U}}\langle\nabla G(u^{k}),u\rangle+J(u)+\langle q^k,\nabla\Omega(u^k)u+\Phi(u)\rangle+\frac{D(u^k,u)}{\epsilon^k},\label{FBS-VAPP-1}\\
p^{k+1}&=&\Pi\left(p^k+\gamma\Theta(u^{k+1})\right),\label{FBS-VAPP-2}
\end{eqnarray}
where $q^k=\Pi\left(p^k+\gamma\Theta(u^k)\right)$. From the strong convexity of $K$, $u^{k+1}$ is unique optimizer of the minimization~\eqref{FBS-VAPP-1}. Notice that, the scheme~\eqref{FBS-VAPP-1}-\eqref{FBS-VAPP-2} exactly coincides with the VAPP algorithm for solving (P).
\subsection{Connections between VAPP and other primal-dual algorithms}
Generally speaking, the majority of existing primal-dual splitting algorithms for convex optimization problems are proposed to solve convex optimization without constraints or just with linear constraints. To discover the connections between VAPP and other primal-dual algorithms, we consider a standard composite optimization problem
\begin{equation}\label{prob:1}
\min_{u} f(Au)+g(u),\quad A\in\RR^{m\times n}
\end{equation}
which can be reformulated as the equality constrained problem
\begin{equation}\label{prob:2}
\begin{array}{cc}
\min\limits_{u,v} &f(v)+g(u)\\
{\rm s.t.} &Au-v=0
\end{array}
\end{equation}
Various primal-dual splitting methods are exploited to sovle problems~\eqref{prob:1}-\eqref{prob:2} by basic splitting scheme. Figure~\ref{fig:2}
and the following statements are used to explain the relationship between VAPP and other primal-dual splitting methods. We focus on connection between VAPP and the primal-dual splitting for constrained convex optimizaiton problem.
\begin{itemize}
\item[(i)] VAPP is a nonlinear FBS algorithm for solving nonlinear convex cone optimization problems.
\item[(ii)] An example of VAPP for problem~\eqref{prob:2} with $G=0$ is Algorithm A$_0$ proposed in~\cite{Zhang2011}, when we choose $K(u)=\frac{1}{2}\left(\|u\|_{Q_0}^2+\alpha\|Au-b\|^2\right)$. Furthermore, if $Q_0=\frac{1}{\sigma}I-\alpha A^{\top}A$, then VAPP coincides with the Bregman operator splitting algorithm (BOS) in~\cite{Zhang2010}.
\item[(iii)] Another related algorithm for problem~\eqref{prob:2} is the predictor corrector proximal multiplier method (PCPM)~\cite{ChenTeboulle1994} was developed by Chen and Teboulle. Note that exact version of PCPM can be finded by VAPP with $G=0$, $J=f(v)+g(u)$ and $K(u,v)=\frac{1}{2}\left(\|u\|^2+\|v\|^2\right)$.
\item[(iv)] Again consider problem (P), its Lagrangian function is $L(u,p)=(G+J)(u)+\langle p,\Theta(u)\rangle$. Taking $T(\cdot)$ as the KKT mapping, then we have $T(w)=\left(\begin{array}{l}\partial_uL(u,p)+\mathcal{N}_{\mathbf{U}}(u)\\-\partial_pL(u,p)+\mathcal{N}_{\mathbf{C}^*}(p)\end{array}\right)$. The alternative projection-proximal method of Tseng (1997)~\cite{Tseng1997} yields the following modified proximal Uzawa algorithm to solve (P).
\begin{equation}\label{Tseng}
\left\{\begin{array}{l}q^k=\Pi\left(p^k+\alpha\Theta(u^k)\right)\\ u^{k+1}=\arg\min\limits_{u\in\mathbf{U}}L(u,q^k)+\frac{\|u-u^k\|^2}{2\alpha}\\ p^{k+1}=\Pi\left(p^k+\alpha\Theta(u^{k+1})\right)\end{array}\right.
\end{equation}
For problem (P), we can take $\tilde{J}(u)=G(u)+J(u)$, $\tilde{\Theta}(u)=\Omega(u)+\Phi(u)$, then VAPP with $K(u)=\frac{\|u\|^2}{2}$ yields the same algorithm~\eqref{Tseng}.
\item[(v)] To the best of our knowledge, the relationship between VAPP/PCPM and DRS, FBFS is not clear. Recently, Combettes~\cite{Combettes2018} applying Tseng's FBFS to Lagrangian of problem~\eqref{prob:2}, established a new algorithm that bears a certain resemblance with the algorithm PCPM~\cite{ChenTeboulle1994}.
\end{itemize}
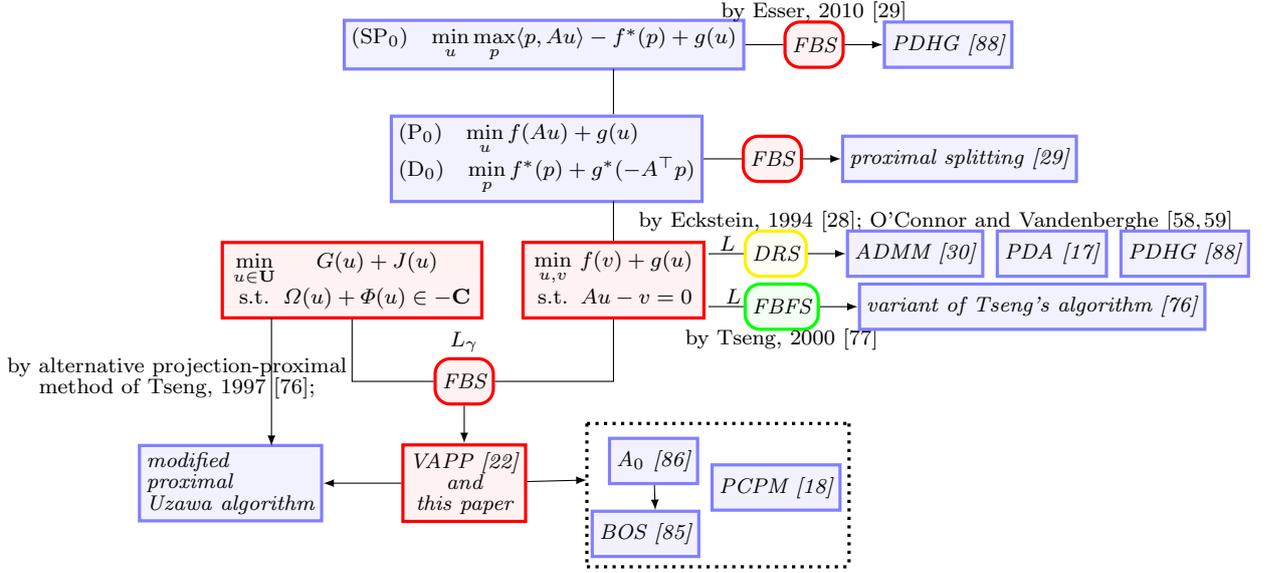
\begin{figure}
\centering
\begin{center}
\scriptsize
		\tikzstyle{format}=[rectangle,draw,thin,fill=white]
		\tikzstyle{test}=[diamond,aspect=2,draw,thin]
		\tikzstyle{point}=[coordinate,on grid,]
\begin{tikzpicture}
[%?????????latex ?????
>=latex,
%???????????
node distance=5mm,
% hv path ???????????????????????????????vh ????skip loop ???
%???-??-??? vskip loop ?????-???-??
 ract/.style={draw=blue!50, fill=blue!5,rectangle,minimum size=6mm, very thick, font=\itshape, align=left},
 racc/.style={rectangle, align=center},
 ractm/.style={draw=red!100, fill=red!5,rectangle,minimum size=6mm, very thick, font=\itshape, align=center},
 cirl/.style={draw, fill=yellow!20,circle,   minimum size=6mm, very thick, font=\itshape, align=center},
 raco0/.style={draw=green!500, fill=green!5,rectangle,rounded corners=2mm,  minimum size=6mm, very thick, font=\itshape, align=center},
 raco1/.style={draw=yellow!500, fill=yellow!5,rectangle,rounded corners=2mm,  minimum size=6mm, very thick, font=\itshape, align=center},
 raco2/.style={draw=red!500, fill=red!5,rectangle,rounded corners=2mm,  minimum size=6mm, very thick, font=\itshape, align=center},
 hv path/.style={to path={-| (\tikztotarget)}},
 vh path/.style={to path={|- (\tikztotarget)}},
 skip loop/.style={to path={-- ++(0,#1) -| (\tikztotarget)}},
 vskip loop/.style={to path={-- ++(#1,0) |- (\tikztotarget)}}]
        \node (a) [ractm]{\baselineskip=3pt\footnotesize $\begin{array}{cc}\min\limits_{u,v}&f(v)+g(u)\\{\rm s.t.}&Au-v=0\end{array}$};
        \node (aaa1) [ract, below = of a, xshift=15, yshift=-30]{\baselineskip=3pt\footnotesize A$_0$~\cite{Zhang2011}};
        \node (aaa2) [ract, below = of aaa1, xshift=-3, yshift=5]{\baselineskip=3pt\footnotesize BOS~\cite{Zhang2010}};
        \node (aaa3) [ract, right = of aaa1, xshift=-10,yshift=-10]{\baselineskip=3pt\footnotesize PCPM~\cite{ChenTeboulle1994}};
        \node (aa2) [ract, above = of a, xshift=-25]{\baselineskip=3pt\footnotesize {\rm(P$_0$)}\quad$\min\limits_{u}f(Au)+g(u)$\\
                                                     \baselineskip=3pt\footnotesize {\rm(D$_0$)}\quad$\min\limits_{p}f^*(p)+g^*(-A^{\top}p)$};
        \node (aacc) [raco2, right = of aa2]{\baselineskip=3pt\footnotesize FBS};
        \node (aacc1) [ract, right = of aacc]{\baselineskip=3pt\footnotesize proximal splitting~\cite{Esser2010}};
        \node (b) [ract, above = of a, xshift=-26, yshift=50]{\baselineskip=3pt\footnotesize {\rm(SP$_0$)}\quad$\min\limits_{u}\max\limits_{p}\langle p,Au\rangle-f^*(p)+g(u)$};
        \node (bcc) [raco2, right = of b]{\baselineskip=3pt\footnotesize FBS};
        \node (bccc2) [racc, above = of bcc, yshift=-18]{\baselineskip=3pt\footnotesize by Esser, 2010~\cite{Esser2010}};
        \node (bcc1) [ract, right = of bcc]{\baselineskip=3pt\footnotesize PDHG~\cite{ZhuChan}};
        \node (c) [ractm, left = of a]{\baselineskip=3pt\footnotesize $\begin{array}{cc}\min\limits_{u\in\mathbf{U}}&G(u)+J(u)\\{\rm s.t.}&\Omega(u)+\Phi(u)\in-\mathbf{C}\end{array}$};
        \node (cc) [raco2, below = of c, xshift=42]{\baselineskip=3pt\footnotesize FBS};
        \node (ccl) [racc, above= of cc, yshift=-15]{\baselineskip=3pt\footnotesize $L_{\gamma}$};
        \node (cc1) [ractm, below = of cc]{\baselineskip=3pt\footnotesize VAPP~\cite{CohenZ}\\
                                           \baselineskip=3pt\footnotesize and \\
                                           \baselineskip=3pt\footnotesize this paper};
        \node (ccl) [ract, left = of cc1, xshift=-15]{\baselineskip=3pt\footnotesize modified\\
                                          \baselineskip=3pt\footnotesize proximal\\
                                          \baselineskip=3pt\footnotesize Uzawa algorithm};
        \node (ccl1) [racc, above = of ccl, xshift=-20]{\baselineskip=3pt\footnotesize by alternative projection-proximal\\
         \baselineskip=3pt\footnotesize method of Tseng, 1997~\cite{Tseng1997};};
        \node (ac1) [raco1, right = of a, yshift=10]{\baselineskip=3pt\footnotesize DRS};
        \node (accc1) [racc, above = of ac1, yshift=-18, xshift=60]{\baselineskip=3pt\footnotesize by Eckstein, 1994~\cite{Eckstein}; O'Connor and Vandenberghe~\cite{ConnorVandenberghe2014,ConnorVandenberghe2017}};
        \node (acl1) [racc, above= of ac1, xshift=-18, yshift=-25]{\baselineskip=3pt\footnotesize $L$};
        \node (acc1) [ract, right = of ac1]{\baselineskip=3pt\footnotesize ADMM~\cite{ADMM1983}};
        \node (accc1) [ract, right = of acc1,xshift=-10]{\baselineskip=3pt\footnotesize PDA~\cite{ChambollePock16}};
        \node (accc11) [ract, right = of accc1,xshift=-10]{\baselineskip=3pt\footnotesize PDHG~\cite{ZhuChan}};
        \node (ac2) [raco0, right = of a, yshift=-10]{\baselineskip=3pt\footnotesize FBFS};
        \node (accc2) [racc, below = of ac2, yshift=18]{\baselineskip=3pt\footnotesize by Tseng, 2000~\cite{Tseng2000}};
        \node (acl2) [racc, above= of ac2, xshift=-18, yshift=-25]{\baselineskip=3pt\footnotesize $L$};
        \node (acc2) [ract, right = of ac2]{\baselineskip=3pt\footnotesize variant of Tseng's algorithm~\cite{Tseng1997}};
        \path %(a) edge[->] (b)
              (aacc) edge[->] (aacc1)
              (cc) edge[->] (cc1)
              (ac1) edge[->] (acc1)
              (ac2) edge[->] (acc2)
              (bcc) edge[->] (bcc1)
              (b) edge[-] (bcc)
              (aaa1) edge[->] (0.55,-3.05)
              (a) edge[-] (0,1.05)
              (0,2.2) edge[-] (0,2.8)
              (aa2) edge[-] (aacc)
              (1.25,0.35) edge[-] (ac1)
              (1.25,-0.35) edge[-] (ac2)
              (cc1) edge[->] (ccl)
              (-4.5,-0.5) edge[->] (-4.5,-2.2)
              %(0.8,-0.5) edge[->] (0.8,-1.9)
              (cc1) edge[->] (-0.35,-2.65);
        \path (c) edge[-, vh path] (cc)
              (a) edge[-, vh path] (cc);
        \draw[dotted,very thick] (-0.35,-1.9)--(3.1,-1.9);
        \draw[dotted,very thick] (-0.35,-1.9)--(-0.35,-3.8);
        \draw[dotted,very thick] (3.1,-1.9)--(3.1,-3.8);
        \draw[dotted,very thick] (-0.35,-3.8)--(3.1,-3.8);
\end{tikzpicture}
\caption{The connection between VAPP and other primal-dual splitting algorithm.}\label{fig:2}
\end{center}
\end{figure}
\section{Further study to some issues for VAPP scheme and implementation}\label{sec:implimentation}
\subsection{The variant of VAPP under new assumption (H$_4'$) of gradient Lipschitz of function $f_p(u)$}
In Section~\ref{VAPP-a}, we show that Assumption (H$_4$) of gradient Lipschitz of $f_p(u)$ uniformly in $p$ plays an important role for convergence analysis for VAPP (in both convex and strongly convex cases). Observe that if the term $\Omega(u)$ is absent from the constraints of (P) or only linear constraints appear, then (H$_4$) obviously holds and take $B_{\Omega}=0$. For another cases, it's not easy to check if (H$_4$) holds. Now we introduce another assumption (H$_4'$) for $f_p(u)$ as\\
\\
{\bf Assumption} {\rm(H$_4'$)} $\Omega$ is differentiable. For any given $p\in\RR^m$, assume that the derivative of function $f_p(u)=\langle p,\Omega(u)\rangle$ is Lipschitz on $\mathbf{U}$ with constant $\tilde{B}_{\Omega}\|p\|$, such that
$$\forall u,v\in \mathbf{U}, \|\nabla f_p(u)-\nabla f_p(v)\|\leq\tilde{B}_{\Omega}\|p\|\cdot\|u-v\|.$$
Next lemma shows that (H$_4'$) holds under the mild condition.
\begin{lemma}
Suppose $\Omega(u)=\left(\Omega_1(u),\ldots,\Omega_m(u)\right)^{\top}$, function $\Omega_j:\RR^n\rightarrow\RR$, $j\in\langle 1,m\rangle$ has Lipschitz gradient with constant $B_{\Omega_j}$. Then $\forall u,v\in \mathbf{U}, \forall p\in\RR^m$ we have
\begin{equation}\label{UB-2}
\|\nabla f_p(u)-\nabla f_p(v)\|\leq\|p\|\cdot B_{\Omega}\|u-v\|\quad\mbox{with}\quad\tilde{B}_{\Omega}=\sum\limits_{j=1}^{m}B_{\Omega_j}.
\end{equation}
\end{lemma}
\begin{proof}
For given $p\in\RR^m$, we have that $f_p(u)=\langle p,\Omega(u)\rangle$ and $\nabla f_p(u)=\left(\nabla\Omega(u)\right)^{\top}p$. It follows that
\begin{eqnarray*}
\|\nabla f_p(u)-\nabla f_p(v)\|&=&\|(\nabla\Omega(u)-\nabla\Omega(v))^{\top}p\|\\
&\leq&|p_1|\cdot\|\nabla\Omega_1(u)-\nabla\Omega_1(v)\|+\cdots+|p_m|\cdot\|\nabla\Omega_m(u)-\nabla\Omega_m(v)\|\\
&\leq&\|p\|\cdot\sum_{j=1}^{m}B_{\Omega_j}\|u-v\|=\|p\|\cdot\tilde{B}_{\Omega}\|u-v\|.\qquad\qquad\qed
\end{eqnarray*}
\end{proof}
It is easy to show that assumption (H$_4'$) implies (H$_4$) with $B_{\Omega}=M\tilde{B}_{\Omega}$ whenever $\|p\|\leq M$. This fact encourage us to propose the following modified VAPP schemes.
\begin{itemize}
\item[{\rm(i)}] For convex problem (P):
\begin{eqnarray*}
\begin{array}{l}
\mbox{\bf VAPP-M Algorithm:}\\
\left\{\begin{array}{l}
u^{k+1}\leftarrow\min\limits_{u\in \mathbf{U}}\langle\nabla G(u^{k}), u\rangle + J(u)+ \langle q^k, \nabla\Omega(u^k)u+\Phi(u)\rangle+\frac{1}{\epsilon^k}D(u,u^k);\\
p^{k+1}\leftarrow \Pi_M\big{(}p^{k}+\rho\Theta(u^{k+1})\big{)}
\end{array}
\right.
\end{array}
\end{eqnarray*}
with $q^k=\Pi_M\big{(}p^{k}+\rho\Theta(u^{k})\big{)}$.
\item[{\rm(ii)}] For strongly convex problem (P)
\begin{eqnarray*}
\begin{array}{l}
\mbox{\bf VAPP-SM Algorithm:}\\
\left\{\begin{array}{l}
u^{k+1}\leftarrow\min\limits_{u\in \mathbf{U}}\langle\nabla G(u^{k}), u\rangle + J(u)+ \langle\tilde{q}^k, \nabla\Omega(u^k)u+\Phi(u)\rangle+\frac{1}{2\epsilon^k}\|u-u^k\|^2;\\
p^{k+1}\leftarrow \Pi_M\big{(}p^{k}+\rho^k\Theta(u^{k+1})\big{)}
\end{array}
\right.
\end{array}
\end{eqnarray*}
with $\tilde{q}^k=\Pi_M\big{(}p^{k}+\rho^k\Theta(u^{k})\big{)}$.
\end{itemize}
Let $M_0$ be a bound of dual optimal solution of (P), denote $M = M_0+1$. Let $\mathfrak{B}_M=\{p|\|p\|\leq M\}$. The estimation of $M_0$ can be found in subsection~\ref{implementation}. By using the projection $\Pi_M(\cdot)$ onto $\mathbf{C}^*\cap\mathfrak{B}_M$. Using the similar arguments in Section~\ref{VAPP-a}, we can also establish the convergence and convergence rate results for VAPP-M and VAPP-SM under the new assumption (H$_4'$). All the assertions of Lemma~\ref{lemma:bound1}, Theorems~\ref{theo:general-convergence},~\ref{thm:ergodic_iteration_complexity},~\ref{theo_2}, and Lemma~\ref{lemma:abound1}, Theorems~\ref{theo:aconvergence},~\ref{theo:aconvergence_rate} are still valid both to VAPP-M and VAPP-SM. Here we omit the details of proof.
\subsection{Issues in the implementation of VAPP for NCCP}\label{implementation}
In this section, we provide three issues in the implementation of VAPP for NCCP: backtracking technique, $\mathbf{C}$-convexity of structured mapping and estimation of the bound for dual optimal solution.
\subsubsection{VAPP with backtracking}\label{VAPPB}
To guarantee the convergence and convergence rate of VAPP, we require that the parameters satisfy the convergence condition~\eqref{eq:parameter} for (P). However, the Lipschitz constant $B_G$, $\tau$ and $B_{\Omega}$ are not always known or computable, thus we must conservatively choose $\{\epsilon^k\}$. This difficulty is stated by industry for implementation of VAPP~\cite{APPturning1,APPturning2}. Recall that the quantity $\Delta^k(u^{k},u^{k+1})$ and the non-increasing $\epsilon^k$ play key role in the convergence and convergence rate analysis. $\Delta^k(u^k,u^{k+1})$ must satisfy the following inequality:\\
$\Delta^k(u^k,u^{k+1})\geq\frac{\beta-\epsilon^k(B_G+B_{\Omega}+\gamma\tau^2)}{2}\|u^k-u^{k+1}\|^2$.\\
This fact furnishes that if $\Delta^k(u^k,u^{k+1})<0$, the satisfication constraint $\epsilon^k<\frac{\beta}{B_G+B_{\Omega}+\gamma\tau^2}$ falls. Based on this fact, we establish the backtracking strategy as follows:\\
\noindent\rule[0.25\baselineskip]{\textwidth}{1.5pt}
{\bf VAPP with Backtracking}\\
\noindent\rule[0.25\baselineskip]{\textwidth}{0.5pt}
{\bf Step 0.} Take $\epsilon^0>0$, $\gamma>0$, $0<\eta<1$, $u^0 \in \mathbf{U}$ and $p^0\in \mathbf{C^*}$. \\
{\bf Step k.} ($k\geq1$) Find the smallest nonnegative integers $i_k$ such that
\begin{equation}
\Delta^k(u^{k-1},\tilde{u})\geq0,
\end{equation}
with $\tilde{\epsilon}=\eta^{i_k}\epsilon^{k-1}$\\
and $\tilde{u}=\arg\min\limits_{u\in U}\langle\nabla G(u^{k-1}), u \rangle + J(u) + \langle q^{k-1}, \nabla\Omega(u^{k-1})u+\Phi(u)\rangle +\frac{1}{\tilde{\epsilon}}D(u,u^{k-1})$.\\
Set $\epsilon^k=\tilde{\epsilon}$ and $u^{k}=\tilde{u}$.\\
Compute $p^{k}=\Pi\big{(}p^{k-1}+\gamma\Theta(u^{k})\big{)}$.\\
\noindent\rule[0.25\baselineskip]{\textwidth}{1.5pt}
The process of VAPP with backtracking guarantees $\Delta^k(u^k,u^{k+1})$ is non-negative, the parameter $\{\epsilon^k\}$ is non-increasing and $\epsilon^k\geq\frac{\eta\beta}{B_{G}+B_{\Omega}+\gamma\tau^2}$. Moreover, after a finite number of iterations, $\epsilon^k$ remains constant. Therefore, all the convergence and convergence rate analysis are still valid. The backtracking strategy also can be used for VAPP-M. (noted that we must take $\Pi_M(\cdot)$ to compute $q^{k-1}$ and $p^k$)
\subsubsection{$\mathbf{C}$-convexity of structured mapping}\label{c-convex}
\indent First note that the affine mapping $\Theta(u)=Au-b$ is $\mathbf{C}$-convex for any convex cone $\mathbf{C}$. When $\mathbf{C}=\RR_{+}^m$, $\Theta(u)$ is $\mathbf{C}$-convex if its elements are convex. Although in~\cite{Boyd}, Boyd and Vandenberghe presented some conditions for $\mathbf{C}$-convexity of a mapping (or convexity with respect to general inequalities), it is generally difficult to verify the $\mathbf{C}$-convexity of mapping $\Theta(u)$ directly.
The following lemma gives the $\mathbf{C}$-convexity of some structured mapping. Their $\mathbf{C}$-convexity allows us to cover some popular applications.
\begin{lemma}\label{lemma:C_convex}
Let $g_0(u)$ be convex on $\RR^n$ and $g(u)$ be a vector function, $g(u)=\big{(}g_1(u),...,g_l(u)\big{)}^{\top}$ whose components $g_j(u)$ are convex on $\RR^n$. Let $Q=[Q_{ij}]_{m\times l}$ be a nonegative matrix and $\omega=(\omega_1,...,\omega_l)^{\top}\in\RR^{l}$ be a nonegative vector with $\omega_j\geq\sum\limits_{i=1}^{m}Q_{ij}$, $j=1,...,l$. Let $A$ be $m'\times n$ matrix and $b\in\RR^{m'}$. Consider $\nu$-norm cone $\mathcal{K}_{\nu}^{k}=\{x=(x_0,\overline{x})\in \RR\times\RR^{k-1}|x_0\geq\|\overline{x}\|_{\nu}\}\subset\RR^k(\nu\geq1)$. Then the following statements hold:
\begin{itemize}
\item[{\rm(i)}] $\Theta(u)=\left(\begin{array}{l}\omega^{\top}g(u)+g_0(u)\\Qg(u)\end{array}\right)$ is $\mathcal{K}_{\nu}^{m+1}$-convex on $\RR^n$;
\item[{\rm(ii)}] $\Theta(u)=\left(\begin{array}{l}g_0(u)\\Au-b\end{array}\right)$ is $\mathcal{K}_{\nu}^{m'+1}$-convex on $\RR^n$;
\item[{\rm(iii)}] $\Theta(u)=\left(\begin{array}{l}\omega^{\top}g(u)+g_0(u)\\Qg(u)\\Au-b\end{array}\right)$ is $\mathcal{K}_{\nu}^{m+m'+1}$-convex on $\RR^n$.
\end{itemize}
\end{lemma}
\begin{proof}
(i) For the sake of brevity, $\forall u,v\in\mathbf{R}^n$, $\alpha\in[0,1]$, denote $\tilde{g}(u,v)=g\big{(}\alpha u+(1-\alpha)v\big{)}-\alpha g(u)-(1-\alpha)g(v)$ and $\tilde{g}_j(u,v)=g_j\big{(}\alpha u+(1-\alpha)v\big{)}-\alpha g_j(u)-(1-\alpha)g_j(v)$, $j=0,1,...,l$.\\
Since $g_j(\cdot)$, $j=0,1,...,l$ are convex, we have $\tilde{g}_j(u,v)\leq 0$, $\forall u,v\in\RR^n$. We observe that
\begin{eqnarray}
\|Q\tilde{g}(u,v)\|_{\nu}&\leq&\|Q\tilde{g}(u,v)\|_{1}\quad\mbox{(since $\nu\geq1$)}\nonumber\\
&\leq&\sum\limits_{i=1}^{m}\sum\limits_{j=1}^{l}\big{|}Q_{ij}\tilde{g}_j(u,v)\big{|}\nonumber\\
&=&\sum\limits_{j=1}^{l}\sum\limits_{i=1}^{m}Q_{ij}\big{|}\tilde{g}_j(u,v)\big{|}\quad\mbox{($Q_{ij}\geq 0$, $i=1,...,m$, $j=1,...,l$)}\nonumber\\
&\leq&\sum\limits_{j=1}^{l}\omega_{j}\big{|}\tilde{g}_j(u,v)\big{|}\quad\mbox{($\omega_j\geq\sum\limits_{i=1}^{m}Q_{ij}$, $j=1,...,l$)}\nonumber\\
&=&-\sum\limits_{j=1}^{l}\omega_{j}\tilde{g}_j(u,v)\quad\mbox{($\tilde{g}_j(u,v)\leq 0$ and $\omega_j\geq 0$, $j=1,...,l$)}\nonumber\\
&\leq&-\big{(}\omega^{\top}\tilde{g}(u,v)+\tilde{g}_0(u,v)\big{)},\qquad\mbox{($\tilde{g}_0(u,v)\leq 0$)}
\end{eqnarray}
which implies that $\Theta\big{(}\alpha u+(1-\alpha)v\big{)}-\alpha\Theta(u)-(1-\alpha)\Theta(v)\in-\mathcal{K}_{\nu}^{m+1}$ and $\Theta(u)$ is $\mathcal{K}_{\nu}^{m+1}$-convex on $\RR^n$.\\
(ii) Statements (ii) and (iii) are directly deduced from statement (i).\qed
\end{proof}
\subsubsection{Estimation of the bound for dual optimal solution}\label{sec:dual_bound}
The estimation of bound $M$ (or $M_0$) is required for implementation of VAPP. In this section, we will provide the estimate of dual optimal bound for problem (P) with special convex cone $\mathbf{C}=\RR_+^m$ or $\mathbf{C}=\mathcal{K}_{\nu}^m$. If $\mathbf{C}=\RR_+^m$, Hiriart-Urruty and Lemar\'echal gave a dual optimal bound as follows. (See Section 2.3 Chapter VII of~\cite{Slaters})
\begin{eqnarray*}
\|p^*\|\leq M_0=\frac{(G+J)(\hat{u})-\underline{G+J}}{\min\limits_{1\leq j\leq m}\{-\Theta_j(\hat{u})\}}.
\end{eqnarray*}
where $\underline{G+J}$ is the lower bound of $(G+J)(u^*)$ and $\hat{u}$ is a vector that satisfies CQC condition for problem (P).\\
\indent When $\mathbf{C}=\mathcal{K}_{\nu}^m$, we will give a dual optimal bound, and the following lemma shows that $M_0$ is computable. A more general case for the estimation of the bound can be found in~\cite{Aybat2014}.
\begin{lemma}\label{lemma:M}
If there exists a point $\hat{u}$ satisfying CQC condition for problem (P) and $\mathbf{C}=\mathcal{K}_{\nu}^{m+1}=\{x=(x_0,\overline{x})\in \RR\times\RR^{m}|x_0\geq\|\overline{x}\|_{\nu}\}$, then we have that
\begin{eqnarray}\label{eq:M0}
\|p^*\|\leq M_0= m^{\max\{\frac{\omega-2}{2\omega},0\}}\cdot2^{\frac{1}{\omega}}\cdot\frac{(G+J)(\hat{u})-\underline{G+J}}{-\theta_0-\|\overline{\theta}\|_\nu},
\end{eqnarray}
where $\frac{1}{\omega}+\frac{1}{\nu}=1$, $\underline{G+J}$ is the lower bound of $(G+J)(u^*)$ and $\Theta(\hat{u})=\left(\begin{array}{l}\theta_0\\\overline{\theta}\end{array}\right)$.
\end{lemma}
\proof
Take $u=\hat{u}$ in the left hand side of saddle point inequality, we have
\begin{eqnarray}\label{eq:M00}
(G+J)(\hat{u})-\underline{G+J}&\geq&(G+J)(\hat{u})-(G+J)(u^*)\nonumber\\
&\geq&\langle p^*,-\Theta(\hat{u})\rangle\nonumber\\
&=&\|p^*\|\cdot\|\Theta(\hat{u})\|\cdot\cos\alpha,
\end{eqnarray}
where $\alpha$ is the included angle between vector $p^{*}\in\mathbf{C}^*$ and $-\Theta(\hat{u})\in\mathring{\mathbf{C}}$. Since $\mathbf{C}=\mathcal{K}_{\nu}^{m+1}$ then we have that
\begin{eqnarray}\label{eq:M1}
\cos\alpha\geq\min_{q_0=1,\|\overline{q}\|_{\omega}\leq 1}\frac{\langle-\Theta(\hat{u}),q\rangle}{\|q\|\cdot\|\Theta(\hat{u})\|}\geq0,\; \mbox{with}\;q=\left(\begin{array}{l}q_0\\\overline{q}\end{array}\right).
\end{eqnarray}
However $$\|q\|\leq m^{\max\{\frac{\omega-2}{2\omega},0\}}\cdot\|q\|_{\omega}\leq m^{\max\{\frac{\omega-2}{2\omega},0\}}\cdot\left(\|\overline{q}\|_{\omega}^{\omega}+(q_0)^{\omega}\right)^{\frac{1}{\omega}}\leq m^{\max\{\frac{\omega-2}{2\omega},0\}}\cdot2^{\frac{1}{\omega}}.$$
Thus,
\begin{eqnarray}\label{eq:M2}
\cos\alpha&\geq&\frac{-\theta_0+\min\limits_{\|\overline{q}\|_{\omega}\leq 1}\langle-\overline{\theta},\overline{q}\rangle}{m^{\max\{\frac{\omega-2}{2\omega},0\}}\cdot2^{\frac{1}{\omega}}\cdot\|\Theta(\hat{u})\|}\nonumber\\
&\geq&\frac{-\theta_0-\max\limits_{\|\overline{q}\|_{\omega}\leq 1}\langle\overline{\theta},\overline{q}\rangle}{m^{\max\{\frac{\omega-2}{2\omega},0\}}\cdot2^{\frac{1}{\omega}}\cdot\|\Theta(\hat{u})\|}\nonumber\\
&=&\frac{-\theta_0-\|\overline{\theta}\|_\nu}{m^{\max\{\frac{\omega-2}{2\omega},0\}}\cdot2^{\frac{1}{\omega}}\cdot\|\Theta(\hat{u})\|}
\end{eqnarray}
where $\Theta(\hat{u})=\left(\begin{array}{l}\theta_0\\\overline{\theta}\end{array}\right)$.
Taking~\eqref{eq:M00} and~\eqref{eq:M2} together, the desired estimate~\eqref{eq:M0} is provided.
\qed
\section{Empirical Results}\label{sec:numerical}
\indent In this section, we test the implementation of our method for solving the Ivanov-type structured elastic net support vector machine problem~\cite{SENSVM,SENSVM2}. The Ivanov regularization problem is a natural expression of structural risk minimization learning problems~\cite{Vapnik98}. This regularization framework provides the ability to directly handle the empirical risk and the hypothesis space~\cite{Vapnik03,TIM}. In this subsection, we consider the Ivanov-type structured elastic net support vector machine problem~\cite{SENSVM,SENSVM2}. This problem is usually formulated as following nonlinear programming with one inequality constraint (see (SEN-SVM-I)). By the definition of $\nu$-norm cone $\mathcal{K}_{\nu}^{k}=\{x=(x_0,\overline{x})\in \RR\times\RR^{k-1}|x_0\geq\|\overline{x}\|_{\nu}\}\subset\RR^k(\nu\geq1)$. The structured elastic net support vector machine problem can be reformulated as following nonlinear programming with cone constraints (see (SEN-SVM-C)).
\begin{equation*}\label{Prob:SEN-SVM}
\begin{array}{|l|l|}
\hline
\mbox{({\bf SEN-SVM-I}){\bf:}}&\mbox{({\bf SEN-SVM-C}){\bf:}}\\
\hline
\begin{array}{cl}
\min\limits_{u\in\RR^n}&\frac{1}{2}\|Au-b\|^2      \\
\rm {s.t}              &\Theta(u)=\alpha\|u\|_1+(1-\alpha)u^{\top}Qu\leq\delta,
\end{array}
&
\begin{array}{cl}
\min\limits_{u\in\RR^n}&\frac{1}{2}\|Au-b\|^2      \\
\rm {s.t}              &\Omega(u)=\left(\begin{array}{c}(1-\alpha)u^{\top}Qu-\delta\\ \alpha u\end{array}\right)\in-\mathcal{K}_1^{n+1},
\end{array}\\
\hline
\end{array}
\end{equation*}
where $u\in\RR^{n}$; $A\in\RR^{m\times n}$, $b\in\RR^{m}$, $Q\in\RR^{n\times n}$, $Q\succ0$, $\alpha\in(0,1)$, $\delta>0$. By the result of Lemma~\ref{lemma:C_convex}, we have that $\Omega(u)$ is $\mathcal{K}_1^{n+1}$-convex. Moreover, it is easy to see that the feasible point $\hat{u}=\mathbf{0}_n$ satisfies CQC conditions and that $0$ is one lower bound of objective function for both (SEN-SVM-I) and (SEN-SVM-C). Moreover, by Hiriart-Urruty and Lemar\'echal's bound and the bound in Lemma~\ref{lemma:M}, we can get the bound of optimal dual as: $M_1=\frac{1}{2\delta}\|b\|^2+1$ (for (SEN-SVM-I)) and $M_2=\frac{\sqrt{n+1}}{2\delta}\|b\|^2+1$ (for (SEN-SVM-C)). Taking $K(u)=\frac{1}{2}\|u\|^2$, we use the VAPP-M scheme to solve (SEN-SVM-I) and (SEN-SVM-C) as follows:
\begin{equation*}
\begin{array}{|l|l|}
\hline
\mbox{{\bf VAPP-M algorithm for (SEN-SVM-I):}}&\mbox{{\bf VAPP-M algorithm for (SEN-SVM-C):}}\\
\hline
\left\{\begin{array}{l}
u^{k+1}=\arg\min\limits_{u\in\RR^n}\|u\|_1+\frac{1}{2\epsilon^k\alpha q_1^k}\big{\|}u-(u^k-\epsilon^k\zeta_1^k)\big{\|}^2\\
p^{k+1}=\min\bigg{\{}M_1,\max\big{\{}0,p^k+\gamma\Theta(u^{k+1})\big{\}}\bigg{\}}
\end{array}\right.
&
\left\{\begin{array}{l}
u^{k+1}=u^k-\epsilon^k\zeta_2^k\\
p^{k+1}=\Pi_{\mathcal{K}_{\infty}^{n+1}\cap\mathfrak{B}_{M_2}}\left(p^k+\gamma\Omega(u^{k+1})\right)
\end{array}\right.\\
\hline
\end{array}
\end{equation*}
where $q_1^k=\min\bigg{\{}M_1,\max\big{\{}0,p^k+\gamma\Theta(u^k)\big{\}}\bigg{\}}$, $q_2^{k}=\Pi_{\mathcal{K}_{\infty}^{n+1}\cap\mathfrak{B}_{M_2}}\left(p^k+\gamma\Omega(u^{k})\right)$  $\zeta_1^k=A^{\top}(Au^k-b)+(1-\alpha)q_1^k(Q+Q^{\top})u^k$ and $\zeta_2^k=A^{\top}(Au^k-b)+\left(\nabla\Omega(u^k)\right)^{\top}q_2^k$.\\
\indent Additionally, another classical algorithm Mirror-Prox (see~\cite{Mirror-Prox-1,Mirror-Prox-2}) can solve convex-concave saddle point problems associated with (SEN-SVM-C):
\begin{equation*}\label{Prob:SEN-SVM-SP}
\mbox{({\bf SEN-SVM-SP}){\bf:}}\quad\min\limits_{u\in\RR^n}\max\limits_{p\in\mathcal{K}_{\infty}^{n+1}\cap\mathfrak{B}_{M_2}}L(u,p)=\frac{1}{2}\|Au-b\|^2+\langle p,\Omega(u)\rangle
\end{equation*}
The scheme of Mirror-Prox algorithm is as follows:
\begin{equation*}
\begin{array}{l}
\mbox{{\bf Mirror-Prox algorithm for (SEN-SVM-SP):}}\\
\left\{\begin{array}{l}
\tilde{u}^{k}=u^k-\gamma^k\nabla_uL(u^k,p^k)\\
\tilde{p}^{k}=\Pi_{\mathcal{K}_{\infty}^{n+1}\cap\mathfrak{B}_{M_2}}\left(p^k+\gamma^k\nabla_pL(u^k,p^k)\right)\\
u^{k+1}=u^k-\gamma^k\nabla_uL(\tilde{u}^k,\tilde{p}^k)\\
p^{k+1}=\Pi_{\mathcal{K}_{\infty}^{n+1}\cap\mathfrak{B}_{M_2}}\left(p^k+\gamma^k\nabla_pL(\tilde{u}^k,\tilde{p}^k)\right)
\end{array}\right.
\end{array}
\end{equation*}
\indent In this experiment, we compared our method against Mirror-prox on a randomly generated Ivanov-type structured elastic net support vector machine problem. The elements of $A\in\RR^{m\times n}$ are selected i.i.d. from a Gaussian $\mathcal{N}(0,1)$ distribution. $Q=B^{\top}B$. The elements of $B\in\RR^{n\times n}$ are selected i.i.d. from a Gaussian $\mathcal{N}(0,1)$ distribution. To construct a sparse true solution $u^*\in\RR^n$, given the dimension $n$ and sparsity $s$, we select $s$ entries of $u^*$ at random to be nonzero and $\mathcal{N}(0,1)$ normally distributed, and set the rest to zero. The measurement vector $b\in\RR^m$ is obtained by $b=Au^*$. We choose $\alpha=0.4$ and $\delta=\alpha\|u^*\|_1+(1-\alpha)\left(u^*\right)^{\top}Qu^*$ with $m=100$, $n=1000$, and $s=5$ in Figure~\ref{fig:1}. It is obvious that the optimal value of the example is zero. We perform this experiment in MATLAB(R2011b) on a personal computer with an Intel Core i5-6200U CPUs (2.40GHz) and 8.00 GB of RAM.\\
\indent The left-hand graph shows the algorithms, plotting suboptimality versus iteration count. The middle graph indicates the algorithms and plots feasibility value versus iteration count. The right-hand graph plots average computation time per iteration of different algorithms. From Figure~\ref{fig:1}, we have the following conclusions:\\
(1) The left-hand graph and the middle graph of Figure~\ref{fig:1} show that the VAPP-M algorithm can effectively solve SEN-SVM problem in both formulations ((SEN-SVM-I) and (SEN-SVM-C)).\\
(2) The left-hand graph and the middle graph of Figure~\ref{fig:1} show that the total number of iterations required of VAPP-M-SEN-SVM-C is less than Mirror Prox. The total number of iterations required of VAPP-M-SEN-SVM-I is near Mirror-Prox-SEN-SVM-SP.\\
(3) The right-hand graph of Figure~\ref{fig:1} shows computation time per iteration of VAPP-M-SEN-SVM-C is about $1/2$ of Mirror-Prox-SEN-SVM-SP used. The computation time per iteration of VAPP-M-SEN-SVM-I is about $1/4$ of Mirror-Prox used.
\begin{figure}
\begin{minipage}[t]{0.32\linewidth}
\centering
\includegraphics[width=1.65in]{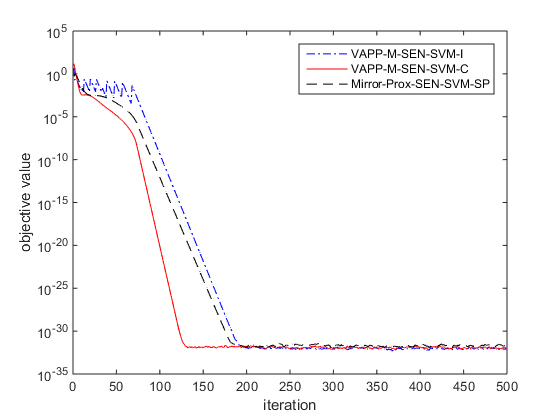}
%\caption{fig1-b}
\label{fig:side:1-b}
\end{minipage}
\begin{minipage}[t]{0.32\linewidth}
\centering
\includegraphics[width=1.65in]{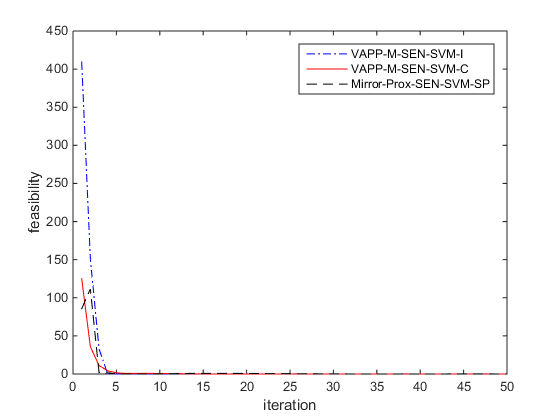}
%\caption{fig1-c}
\label{fig:side:1-c}
\end{minipage}
\begin{minipage}[t]{0.32\linewidth}
\centering
\includegraphics[width=1.65in]{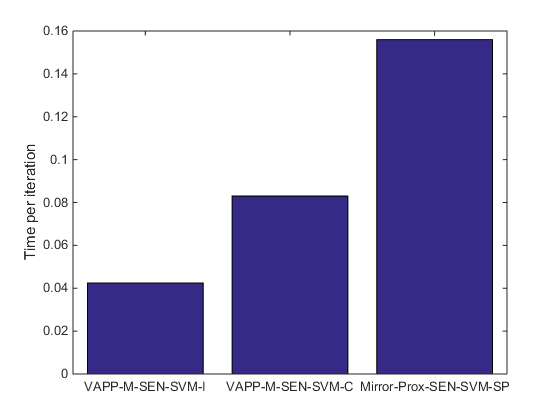}
%\caption{fig2-b}
\label{fig:side:2-b}
\end{minipage}
\caption{$m=100$, $n=1000$, and $s=5$. The left-hand graph shows the algorithms and plots suboptimality versus iteration count. The middle graph indicates the algorithms and plots feasibility value versus iteration count. The right-hand graph plots average computation time per iteration of different algorithms}
\label{fig:1}
\end{figure}
\section{Appendix}
{\bf A$_1$: Proof of Lemma~\ref{lemma:bound1} (Descent inequalities of generalized distance function):}\\
{\sl Step 1. Estimate $L(u^{k+1},q^k)-L(u,q^k)$:}\\
For the primal subproblem~\eqref{primal} of VAPP, the unique solution $u^{k+1}$ is characterized by the following variational inequality:
\begin{eqnarray}\label{eq:VI}
\langle\nabla G(u^{k}),u-u^{k+1}\rangle+J(u)-J(u^{k+1})+\langle q^k,\nabla\Omega(u^{k})(u-u^{k+1})+\Phi(u)-\Phi(u^{k+1})\rangle\nonumber\\
+\frac{1}{\epsilon^k}\langle \nabla K(u^{k+1})-\nabla K(u^k), u-u^{k+1}\rangle\geq 0, \forall u\in\mathbf{U},\nonumber\\
\end{eqnarray}
which follows that
\begin{eqnarray}\label{eq:primal_bound1}
L(u^{k+1},q^k)-L(u,q^k)&=&(G+J)(u^{k+1})-(G+J)(u)+\langle q^k, \Theta(u^{k+1})-\Theta(u)\rangle\qquad\nonumber\\
&\leq&\underbrace{G(u^{k+1})-G(u)+\langle\nabla G(u^{k}),u-u^{k+1}\rangle}_{\Lambda_1}\nonumber\\
&&+\underbrace{\langle q^k,\Omega(u^{k+1})-\Omega(u)+\nabla\Omega(u^{k})(u-u^{k+1})\rangle}_{\Lambda_2}\nonumber\\
&&+\underbrace{\frac{1}{\epsilon^k}\langle\nabla K(u^{k+1})-\nabla K(u^k), u-u^{k+1}\rangle}_{\Lambda_3}.
\end{eqnarray}
By the convexity of $G$, we estimate term $\Lambda_1$ in~\eqref{eq:primal_bound1}.
\begin{eqnarray}\label{eq:primal_bound3}
\Lambda_1&=&G(u^k)-G(u)+\langle\nabla G(u^{k}),u-u^{k}\rangle+\big{(}G(u^{k+1})-G(u^{k})-\langle\nabla G(u^{k}),u^{k+1}-u^{k}\rangle\big{)}\nonumber\\
         &\leq& G(u^{k+1})-G(u^{k})-\langle\nabla G(u^{k}),u^{k+1}-u^{k}\rangle.
\end{eqnarray}
Since $\Omega(u)$ is $\mathbf{C}$-convex, $q^k\in\mathbf{C}^*$, then $\langle q^k,\Omega(u)\rangle$ is convex and
\begin{eqnarray}\label{eq:primal_bound33}
\Lambda_2&=&\langle q^k,\Omega(u^{k})-\Omega(u)+\nabla\Omega(u^{k})(u-u^{k})\rangle+\big{(}\langle q^k,\Omega(u^{k+1})-\Omega(u^{k})-\nabla\Omega(u^{k})(u^{k+1}-u^{k})\rangle\big{)}\nonumber\\
&\leq&\langle q^k,\Omega(u^{k+1})-\Omega(u^{k})-\nabla\Omega(u^{k})(u^{k+1}-u^{k})\rangle.
\end{eqnarray}
Since $K(\cdot)$ satisfies Assumption~\ref{assump2}, simple algebraic operation follows that
\begin{equation}\label{eq:primal_bound2}
\Lambda_3=\frac{1}{\epsilon^k}\langle \nabla K(u^{k+1})-\nabla K(u^k),u-u^{k+1}\rangle=\frac{1}{\epsilon^k}\big{[}D(u,u^k)-D(u,u^{k+1})-D(u^{k+1},u^k)\big{]},
\end{equation}
Take $\Lambda_1$, $\Lambda_2$ and $\Lambda_3$ into~\eqref{eq:primal_bound1}, we have
\begin{eqnarray*}\label{eq:primal_bound4}
L(u^{k+1},q^k)-L(u,q^k)&\leq&\frac{1}{\epsilon^k}D(u,u^k)-\frac{1}{\epsilon^{k}}D(u,u^{k+1})-\frac{1}{\epsilon^k}\bigg{\{}D(u^{k+1},u^{k})\;\nonumber\\
&&-\epsilon^k\bigg{[}\big{(}G(u^{k+1})-G(u^{k})-\langle\nabla G(u^{k}),u^{k+1}-u^{k}\rangle\big{)}\nonumber\\
&&+\langle q^k,\Omega(u^{k+1})-\Omega(u^{k})-\nabla\Omega(u^{k})(u^{k+1}-u^{k})\rangle\bigg{]}\bigg{\}}.
\end{eqnarray*}
Multiply $\epsilon^k$ on both side of the above inequality, and we have that
\begin{eqnarray}\label{eq:EP-bound1}
&&\epsilon^k[L(u^{k+1},q^k)-L(u,q^k)]\nonumber\\
&\leq& D(u,u^k)-D(u,u^{k+1})-\Delta^k(u^k,u^{k+1})-\frac{\epsilon^k\gamma}{2}\|\Theta(u^{k})-\Theta(u^{k+1})\|^2.
\end{eqnarray}
{\sl Step 2. Estimate $L(u^{k+1},p)-L(u^{k+1},q^k)$:}\\
We first derive two inequalities. By the property of projection~\eqref{eq:Projecproperty3} with $u=p^k+\gamma \Theta(u^{k+1})$, $v=p$, $\forall p\in\mathbf{C}^*$, we have
\begin{equation}\label{proj1}
\frac{1}{\gamma}\langle p-p^{k+1}, p^k+\gamma \Theta(u^{k+1})-p^{k+1}\rangle\leq 0.
\end{equation}
Using Proposition~\ref{proposition} with $u=\gamma\Theta(u^{k+1})$, $v=\gamma\Theta(u^k)$, and $w=p^k$, we have
\begin{equation}\label{proj2}
2\langle p^{k+1}-q^k,\gamma\Theta(u^{k+1})\rangle\leq\|\gamma\Theta(u^{k+1})-\gamma\Theta(u^k)\|^2+\|p^{k+1}-p^k\|^2-\|q^k-p^k\|^2.
\end{equation}
Statement (ii) follows from~\eqref{proj1} and~\eqref{proj2}:
\begin{eqnarray}\label{eq:VI_solution2}
&&L(u^{k+1},p)-L(u^{k+1},q^k)\nonumber\\
&=&\langle p-q^k,\Theta(u^{k+1})\rangle\nonumber\\
&=&\langle p-p^{k+1},\Theta(u^{k+1})\rangle+\langle p^{k+1}-q^k,\Theta(u^{k+1})\rangle\nonumber\\
&=&\frac{1}{\gamma}\langle p-p^{k+1},p^k+\gamma\Theta(u^{k+1})-p^{k+1}\rangle+\frac{1}{\gamma}\langle p-p^{k+1},p^{k+1}-p^k\rangle+\langle p^{k+1}-q^k,\Theta(u^{k+1})\rangle\nonumber\\
&\leq&\frac{1}{\gamma}\langle p-p^{k+1},p^{k+1}-p^k\rangle+\langle p^{k+1}-q^k,\Theta(u^{k+1})\rangle\qquad\qquad\qquad\qquad\mbox{(by inequality~\eqref{proj1})}\nonumber\\
&\leq&\frac{1}{\gamma}\langle p-p^{k+1},p^{k+1}-p^{k}\rangle+\frac{1}{2\gamma}\|p^k-p^{k+1}\|^2-\frac{1}{2\gamma}\|q^k-p^k\|^2+\frac{\gamma}{2}\|\Theta(u^{k})-\Theta(u^{k+1})\|^2\nonumber\\
&&\qquad\qquad\qquad\qquad\qquad\qquad\qquad\qquad\qquad\qquad\qquad\qquad\qquad\qquad\;\mbox{(by inequality~\eqref{proj2})}\nonumber\\
&=&\frac{1}{2\gamma}\big{[}\|p-p^{k}\|^2-\|p-p^{k+1}\|^2\big{]}-\frac{1}{2\gamma}\|q^k-p^k\|^2+\frac{\gamma}{2}\|\Theta(u^{k})-\Theta(u^{k+1})\|^2
\end{eqnarray}
Then, multiplying $\epsilon^k$ on both side of~\eqref{eq:VI_solution2}, we obtain
\begin{eqnarray}\label{eq:EP-bound2}
&&\epsilon^k[L(u^{k+1},p)-L(u^{k+1},q^k)]\nonumber\\
&=&\frac{\epsilon^k}{2\gamma}\big{[}\|p-p^{k}\|^2-\|p-p^{k+1}\|^2\big{]}-\frac{\epsilon^k}{2\gamma}\|q^k-p^k\|^2+\frac{\epsilon^k\gamma}{2}\|\Theta(u^{k})-\Theta(u^{k+1})\|^2\nonumber\\
&\leq&\frac{\epsilon^k}{2\gamma}\|p-p^{k}\|^2-\frac{\epsilon^{k+1}}{2\gamma}\|p-p^{k+1}\|^2-\frac{\epsilon^{k}}{2\gamma}\|q^k-p^k\|^2+\frac{\epsilon^k\gamma}{2}\|\Theta(u^{k})-\Theta(u^{k+1})\|^2\nonumber\\
&&\qquad\qquad\qquad\qquad\qquad\qquad\qquad\qquad\qquad\qquad\qquad\quad\mbox{(since $\epsilon^{k+1}\leq\epsilon^k$)}
\end{eqnarray}
{\sl Step 3. Estimate $L(u^{k+1},p)-L(u,q^k)$:}\\
Summing~\eqref{eq:EP-bound1} and~\eqref{eq:EP-bound2}, the desired result is coming.\qed
\qquad\\
{\bf A$_2$: Proof of Theorem~\ref{theo:general-convergence} (Convergence analysis for VAPP)}\\
Take $u=u^{*}$ and $p=p^*$ in Lemma~\ref{lemma:bound1}, then we have that
\begin{eqnarray}\label{nonegative-measure-func}
&&\big{[}D(u^*,u^{k+1})+\frac{\epsilon^{k+1}}{2\gamma}\|p^*-p^{k+1}\|^2\big{]}-\big{[}D(u^*,u^k)+\frac{\epsilon^k}{2\gamma}\|p^*-p^{k}\|^2\big{]}\nonumber\\
&\leq&\epsilon^k[L(u^*,q^k)-L(u^{k+1},p^*)]-\big{[}\Delta^k(u^k,u^{k+1})+\frac{\epsilon^k}{2\gamma}\|q^k-p^k\|^2\big{]}\nonumber\\
&\leq&-\big{[}\Delta^k(u^k,u^{k+1})+\frac{\epsilon^k}{2\gamma}\|q^k-p^k\|^2\big{]}\quad\mbox{(since $(u^*,p^*)$ is a saddle point~\eqref{VIS})}\nonumber\\
&\leq&-\bigg{[}\frac{\beta-\epsilon^k(B_G+B_{\Omega}+\gamma\tau^2)}{2}\|u^{k}-u^{k+1}\|^2+\frac{\epsilon^k}{2\gamma}\|q^k-p^k\|^2\bigg{]}\nonumber\\
&&\qquad\qquad\qquad\qquad\quad\;\mbox{(from~\eqref{eq:bounddelta},  $\Delta^k(u,v)\geq\frac{\beta-\epsilon^k(B_G+B_{\Omega}+\gamma\tau^2)}{2}\|u-v\|^2$)}\nonumber\\
&\leq&-\bigg{[}\frac{\beta-\overline{\epsilon}(B_G+B_{\Omega}+\gamma\tau^2)}{2}\|u^{k}-u^{k+1}\|^2+\frac{\underline{\epsilon}}{2\gamma}\|q^k-p^k\|^2\bigg{]}.\\
&&\qquad\qquad\qquad\qquad\qquad\qquad\qquad\qquad\qquad\mbox{(since $\underline{\epsilon}\leq\epsilon^k\leq\overline{\epsilon}$ satisfy~\eqref{eq:parameter})}\nonumber
\end{eqnarray}
Since $\{\epsilon^k\}$ satisfies~\eqref{eq:parameter}, we conclude that the sequence $\{D(u^*,u^{k})+\frac{\epsilon^{k}}{2\gamma}\|p^*-p^{k}\|^2\}$ is strictly decreasing, unless $u^k=u^{k+1}$ and $p^k=q^k$ or $p^k=p^{k+1}$. The rest of proof is similar to that of~\cite{CohenZ}.\qed
\qquad\\
{\bf A$_3$: Proof of Theorem~\ref{thm:ergodic_iteration_complexity} (Bifunction value estimation, primal suboptimality and feasibility for solving (P) by VAPP)}\\
{\rm(i)} Note that the set $\mathbf{U}\times\mathbf{C}^*$ is convex, and the VAPP scheme guarantees that $(u^k,p^k)\in\mathbf{U}\times\mathbf{C}^*$, $\forall k\in\mathbb{N}$; thus we have $(\bar{u}_t,\bar{p}_t)\in\mathbf{U}\times\mathbf{C}^*$. Since $\{\epsilon^k\}$ satisfies~\eqref{eq:parameter}, then $\Delta^k(u^k,u^{k+1})\geq0$. From Lemma~\ref{lemma:bound1}, we have
$$\epsilon^k[L(u^{k+1},p)-L(u,q^k)]\leq\big{[}D(u,u^k)+\frac{\epsilon^k}{2\gamma}\|p-p^{k}\|^2\big{]}-\big{[}D(u,u^{k+1})+\frac{\epsilon^{k+1}}{2\gamma}\|p-p^{k+1}\|^2\big{]}.$$
Note that the bifunction $L(u',p)-L(u,p')$ is convex in $u'$ and linear in $p'$ for given $u\in\mathbf{U}$, $p\in\mathbf{C}^*$. Summing the above inequality over $k=0,1,\ldots,t$, we obtain that
\begin{eqnarray*}
L(\bar{u}_{t},p)-L(u,\bar{p}_t)&\leq&\frac{1}{\sum_{k=0}^{t}\epsilon^k}\sum_{k=0}^{t}\epsilon^k[L(u^{k+1},p)-L(u,q^k)]\\
&\leq&\frac{1}{\underline{\epsilon}(t+1)}\big{[}D(u,u^0)+\frac{\epsilon^0}{2\gamma}\|p-p^{0}\|^2\big{]},\;\forall u\in\mathbf{U},p\in\mathbf{C}^*.
\end{eqnarray*}
{\rm(ii)} If $\|\Pi(\Theta(\bar{u}_t))\|=0$, statement (ii) is obviously true.\\
 Otherwise, taking $u=u^*\in\mathbf{U}$ and $p=\hat{p}=\frac{(M_0+1)\Pi(\Theta(\bar{u}_t))}{\|\Pi(\Theta(\bar{u}_t))\|}\in\mathbf{C}^*\cap\mathfrak{B}_M$ in statement (i) of this theorem, we have that
\begin{eqnarray}\label{eq:rate1}
&&L(\bar{u}_t,\hat{p})-L(u^*,\bar{p}_t)\nonumber\\
&=&(G+J)(\bar{u}_{t})-(G+J)(u^*)+\langle\frac{(M_0+1)\Pi(\Theta(\bar{u}_t))}{\|\Pi(\Theta(\bar{u}_t))\|}, \Theta(\bar{u}_{t})\rangle-\langle\bar{p}_t, \Theta(u^*)\rangle\nonumber\\
&\geq&(G+J)(\bar{u}_{t})-(G+J)(u^*)+\langle\frac{(M_0+1)\Pi(\Theta(\bar{u}_t))}{\|\Pi(\Theta(\bar{u}_t))\|}, \Theta(\bar{u}_{t})\rangle\;\;\mbox{(since $\langle\bar{p}_t, \Theta(u^*)\rangle\leq0$)}\nonumber\\
&=&(G+J)(\bar{u}_{t})-(G+J)(u^*)+\langle\frac{(M_0+1)\Pi(\Theta(\bar{u}_t))}{\|\Pi(\Theta(\bar{u}_t))\|}, \Pi(\Theta(\bar{u}_t))+\Pi_{-\mathbf{C}}(\Theta(\bar{u}_t))\rangle\nonumber\\
&&\qquad\qquad\qquad\qquad\qquad\qquad\qquad\qquad\qquad\qquad\qquad\quad\quad\;\;\;\;\;\mbox{(from~\eqref{eq:Projecproperty5})}\nonumber\\
&=&(G+J)(\bar{u}_{t})-(G+J)(u^*)+(M_0+1)\|\Pi(\Theta(\bar{u}_t))\|.\qquad\qquad\qquad\mbox{(from~\eqref{eq:Projecproperty6})}
\end{eqnarray}
Combining statement (i) of this theorem,~\eqref{eq:rate1} yields that
\begin{eqnarray}\label{eq:rate3}
(G+J)(\bar{u}_{t})-(G+J)(u^*)+(M_0+1)\|\Pi(\Theta(\bar{u}_t))\|&\leq&\frac{D(u^*,u^0)+\frac{\epsilon^0}{2\gamma}\|\hat{p}-p^{0}\|^2}{\underline{\epsilon}(t+1)}\nonumber\\
                                                         &\leq&\frac{d_1}{\underline{\epsilon}(t+1)},
\end{eqnarray}
where $d_1=\max\limits_{\|p\|\leq M_0+1}\big{[}D(u^*,u^0)+\frac{\epsilon^0}{2\gamma}\|p-p^{0}\|^2\big{]}$. Moreover, taking $u=\bar{u}_t$ in the right hand side of saddle point inequality~\eqref{saddle point:L} yields that
\begin{eqnarray}\label{eq:rate4}
(G+J)(\bar{u}_{t})-(G+J)(u^*)&\geq&-\langle p^*,\Theta(\bar{u}_t)\rangle\nonumber\\
                             &=&-\langle p^*,\Pi(\Theta(\bar{u}_t))+\Pi_{-\mathbf{C}}(\Theta(\bar{u}_t))\rangle\qquad\mbox{(since~\eqref{eq:Projecproperty5})}\nonumber\\
                             &\geq&-\langle p^*,\Pi(\Theta(\bar{u}_t))\rangle\quad\mbox{(since $\langle p^*,\Pi_{-\mathbf{C}}(\Theta(\bar{u}_t))\rangle\leq 0$)}\nonumber\\
                             &\geq&-\|p^*\|\|\Pi(\Theta(\bar{u}_t))\|\nonumber\\
                             &\geq&-M_0\|\Pi(\Theta(\bar{u}_t))\|.\quad\;\;\;\mbox{(by $\|p^*\|\leq M_0$)}
\end{eqnarray}
Taking~\eqref{eq:rate3} and~\eqref{eq:rate4} together, we get that $\|\Pi(\Theta(\bar{u}_t))\|\leq\frac{d_1}{\underline{\epsilon}(t+1)}$.\\
{\rm(iii)} Since $(M_0+1)\|\Pi(\Theta(\bar{u}_t))\|\geq 0$, from~\eqref{eq:rate3} we have\\
$$(G+J)(\bar{u}_{t})-(G+J)(u^*)\leq\frac{d_1}{\underline{\epsilon}(t+1)}.$$\\
Combining statement (ii) of this theorem and~\eqref{eq:rate4}, we obtain that\\
$$(G+J)(\bar{u}_{t})-(G+J)(u^*)\geq-\frac{M_0d_1}{\underline{\epsilon}(t+1)}.$$
\qquad\\
{\bf A$_4$: Proof of Lemma~\ref{lemma:ALBounded}:}\\
Suppose the assertion of the lemma does not hold, that is, for any $\kappa>0$, there is $\|p^j\|\leq d_p$ so that all optimizers $\hat{u}(p^j)\in\arg\min\limits_{u\in\mathbf{U}}L_\gamma(u,p^j)$ satisfy $\|\hat{u}(p^j)\|>\kappa$. Then, we construct a sequence $\{\hat{u}(p^j)\}$ such that $\|\hat{u}(p^j)\|\rightarrow +\infty$.\\
\indent On the other hand, we observe that
\begin{eqnarray*}
L_\gamma(\hat{u}(p^j),p^j)&=&(G+J)(\hat{u}(p^j))+\varphi\big{(}\Theta(\hat{u}(p^j)),p^j\big{)}\\
                          &=&(G+J)(\hat{u}(p^j))+\max_{q\in\mathbf{C}^*}\langle q,\Theta(\hat{u}(p^j))\rangle-\frac{1}{2\gamma}\|q-p^j\|^2\\
                          &\geq&(G+J)(\hat{u}(p^j))-\frac{1}{2\gamma}\|p^j\|^2\\
                          &\geq&(G+J)(\hat{u}(p^j))-\frac{d_p^2}{2\gamma}.
\end{eqnarray*}
\indent Since $\|\hat{u}(p^j)\|\rightarrow+\infty$, from the coercivity of $(G+J)(u)$, we have $\psi_\gamma(p^j)=L_\gamma(\hat{u}(p^j),p^j)\rightarrow+\infty$. However, from the boundness of $\{p^j\}$ and the continuity of $\psi_\gamma(\cdot)$, we conclude that $\psi_\gamma(p^j)$ is bounded, which follows one contradiction and assertion of lemma is provided.\qed
\qquad\\
\noindent{\bf A$_5$: Proof of Theorem~\ref{theo_2} (Approximate saddle point and dual suboptimality for solving (P) by VAPP):}\\
{\rm(i)} From statement (i) of Theorem~\ref{thm:ergodic_iteration_complexity}, it is easy to have that, for any $(u,p)\in(\mathbf{U}\cap\mathfrak{B}^{u})\times(\mathbf{C}^*\cap\mathfrak{B}^{p})$,
\begin{eqnarray}\label{saddle:L0}
L(\bar{u}_{t},p)-L(u,\bar{p}_{t})\leq
\frac{D(u,u^0)+\frac{\epsilon^0}{2\gamma}\|p-p^{0}\|^2}{\underline{\epsilon}(t+1)}\leq\frac{d_2}{\underline{\epsilon}(t+1)}
\end{eqnarray}
where $d_2=\max_{(u,p)\in(\mathbf{U}\cap\mathfrak{B}^{u})\times(\mathbf{C}^*\cap\mathfrak{B}^{p}))}\big{[}D(u,u^0)+\frac{\epsilon^0}{2\gamma}\|p-p^{0}\|^2\big{]}$.\\
Since $\bar{u}_t\in\mathbf{U}\cap\mathfrak{B}^{u}$, then taking $u=\bar{u}_t$ in~\eqref{saddle:L0}, we obtain
\begin{eqnarray}\label{saddle:L1}
    L(\bar{u}_{t},p)-L(\bar{u}_{t},\bar{p}_{t})\leq\frac{d_2}{\underline{\epsilon}(t+1)}, \forall p\in\mathbf{C}^*\cap\mathfrak{B}^{p}.
\end{eqnarray}
Similarly, by taking $p=\bar{p}_t\in\mathbf{C}^*\cap\mathfrak{B}^{p}$ in~\eqref{saddle:L0}, we obtain
\begin{eqnarray}\label{saddle:L2}
L(\bar{u}_{t},\bar{p}_t)-L(u,\bar{p}_{t})\leq\frac{d_2}{\underline{\epsilon}(t+1)}, \forall u\in\mathbf{U}\cap\mathfrak{B}^{u}.
\end{eqnarray}
{\rm(ii)} In the left-hand side of inequality in statement (i), taking $p=0$, we get $\langle\bar{p}_t,\Theta(\bar{u}_t)\rangle\geq-\frac{d_2}{\underline{\epsilon}(t+1)}$. Then, from~\eqref{func:varphi_2}, we have
\begin{eqnarray}\label{saddle:Lr1}
\varphi\big{(}\Theta(\bar{u}_t),\bar{p}_t\big{)}\geq\langle\bar{p}_t,\Theta(\bar{u}_t)\rangle\geq-\frac{d_2}{\underline{\epsilon}(t+1)}.
\end{eqnarray}
On the other hand, for $p\in\mathbf{C}^*\cap\mathfrak{B}^{p}$, we have
\begin{eqnarray}\label{saddle:Lr2}
\varphi\big{(}\Theta(\bar{u}_t),p\big{)}&=&\min_{\xi\in-\mathbf{C}}\langle p,\Theta(\bar{u}_t)-\xi\rangle+\frac{\gamma}{2}\|\Theta(\bar{u}_t)-\xi\|^2\qquad\mbox{(from~\eqref{func:varphi_1})}\nonumber\\
&\leq&\langle p,\Theta(\bar{u}_t)-\Pi_{-\mathbf{C}}(\Theta(\bar{u}_t))\rangle+\frac{\gamma}{2}\|\Theta(\bar{u}_t)-\Pi_{-\mathbf{C}}(\Theta(\bar{u}_t))\|^2\nonumber\\
&\leq&\|p\|\cdot\|\Pi(\Theta(\bar{u}_t))\|+\frac{\gamma}{2}\|\Pi(\Theta(\bar{u}_t))\|^2\nonumber\\
&\leq&\frac{r^pd_1}{\underline{\epsilon}(t+1)}+\frac{\gamma(d_1)^2}{2\underline{\epsilon}^2(t+1)^2}.\\
&&\mbox{(from statment (ii) of Theorem~\ref{thm:ergodic_iteration_complexity} and $p\in\mathbf{C}^*\cap\mathfrak{B}^{p}$)}\nonumber
\end{eqnarray}
Therefore, we get the left-hand side of inequality in statement (ii):
\begin{eqnarray}\label{eq:Lr_3}
L_{\gamma}(\bar{u}_{t},p)-L_{\gamma}(\bar{u}_{t},\bar{p}_{t})&=&\varphi(\Theta(\bar{u}_t),p)-\varphi(\Theta(\bar{u}_t),\bar{p}_t)\nonumber\\
                                                             &\leq&\frac{r^pd_1+d_2}{\underline{\epsilon}(t+1)}+\frac{\gamma (d_1)^2}{2\underline{\epsilon}^2(t+1)^2},
\end{eqnarray}
From~\eqref{saddle:Lr1} and~\eqref{saddle:Lr2}, it also has that
\begin{eqnarray*}
-\frac{d_2}{\underline{\epsilon}(t+1)}\leq\langle\bar{p}_t,\Theta(\bar{u}_t)\rangle\leq\varphi(\Theta(\bar{u}_t),\bar{p}_t)\leq\frac{r^p d_1}{\underline{\epsilon}(t+1)}+\frac{\gamma(d_1)^2}{2\underline{\epsilon}^2(t+1)^2},
\end{eqnarray*}
which follows that
\begin{eqnarray*}
\varphi(\Theta(\bar{u}_t),\bar{p}_t)-\langle\bar{p}_t,\Theta(\bar{u}_t)\rangle&\leq&\frac{r^pd_1}{\underline{\epsilon}(t+1)}+\frac{\gamma (d_1)^2}{2\underline{\epsilon}^2(t+1)^2}-(-\frac{d_2}{\underline{\epsilon}(t+1)})\\
&=&\frac{r^pd_1+d_2}{\underline{\epsilon}(t+1)}+\frac{\gamma (d_1)^2}{2\underline{\epsilon}^2(t+1)^2}.
\end{eqnarray*}
Then, for $u\in\mathbf{U}\cap\mathfrak{B}^{u}$, we have
\begin{eqnarray}\label{eq:Lr_4}
L_\gamma(\bar{u}_{t},\bar{p}_{t})&\leq&L(\bar{u}_{t},\bar{p}_{t})+\frac{r^pd_1+d_2}{\underline{\epsilon}(t+1)}+\frac{\gamma (d_1)^2}{2\underline{\epsilon}^2(t+1)^2}\nonumber\\
&\leq&L(u,\bar{p}_{t})+\frac{r^pd_1+2d_2}{\underline{\epsilon}(t+1)}+\frac{\gamma (d_1)^2}{2\underline{\epsilon}^2(t+1)^2}\qquad\mbox{(by right hand side of statement (i))}\nonumber\\
&\leq&L_\gamma(u,\bar{p}_{t})+\frac{r^pd_1+2d_2}{\underline{\epsilon}(t+1)}+\frac{\gamma(d_1)^2}{2\underline{\epsilon}^2(t+1)^2},\qquad\qquad\qquad\mbox{(from~\eqref{func:varphi_2})}
\end{eqnarray}
which follows the right-hand side of inequality in statement (ii).\\
{\rm(iii)} For saddle point $(u^*,p^*)$, we have
\begin{eqnarray}\label{eq:saddlepoint_Lr}
L_{\gamma}(u^*,p)\leq L_{\gamma}(u^*,p^*)\leq L_{\gamma}(u,p^*), \forall u\in\mathbf{U}, p\in\RR^m
\end{eqnarray}
\indent Taking $u=\bar{u}_t$, $p=\bar{p}_t$ in~\eqref{eq:saddlepoint_Lr}, and taking $u=\hat{u}(\bar{p}_t)$, $p=p^*$ in statement (ii) of this theorem, we obtain the following two inequalities, respectively:
\begin{eqnarray*}
L_{\gamma}(u^*,\bar{p}_t)\leq&L_{\gamma}(u^*,p^*)&\leq L_{\gamma}(\bar{u}_t,p^*),\label{eq:dual_1}
\end{eqnarray*}
and
\begin{eqnarray*}
-\frac{r^pd_1+d_2}{\underline{\epsilon}(t+1)}-\frac{\gamma(d_1)^2}{2\underline{\epsilon}^2(t+1)^2}+L_{\gamma}(\bar{u}_{t},p^*)\leq L_{\gamma}(\bar{u}_{t},\bar{p}_{t})\leq L_{\gamma}(\hat{u}(\bar{p}_t),\bar{p}_{t})+\frac{r^pd_1+2d_2}{\underline{\epsilon}(t+1)}+\frac{\gamma (d_1)^2}{2\underline{\epsilon}^2(t+1)^2}.\label{eq:dual_2}
\end{eqnarray*}
Combining these two inequalities, the desired inequality is obtained:
\begin{eqnarray*}
-\frac{r^pd_1+d_2}{\underline{\epsilon}(t+1)}-\frac{\gamma(d_1)^2}{2\underline{\epsilon}^2(t+1)^2}+L_{\gamma}(u^*,p^*)\leq L_{\gamma}(\hat{u}(\bar{p}_t),\bar{p}_{t})+\frac{r^pd_1+2d_2}{\underline{\epsilon}(t+1)}+\frac{\gamma (d_1)^2}{2\underline{\epsilon}^2(t+1)^2}.
\end{eqnarray*}
Therefore
\begin{eqnarray}
\psi_\gamma(p^*)=L_{\gamma}(u^*,p^*)&\leq&L_{\gamma}(\hat{u}(\bar{p}_t),\bar{p}_{t})+\frac{2r^pd_1+3d_2}{\underline{\epsilon}(t+1)}+\frac{\gamma (d_1)^2}{\underline{\epsilon}^2(t+1)^2}\nonumber\\
&=&\psi_\gamma(\bar{p}_t)+\frac{2r^pd_1+3d_2}{\underline{\epsilon}(t+1)}+\frac{\gamma(d_1)^2}{\underline{\epsilon}^2(t+1)^2}.
\end{eqnarray}
\qed

%\begin{acknowledgements}
%\end{acknowledgements}



\end{document}